\documentclass[11pt,twoside]{article}
\usepackage{etex}
\usepackage{anysize}
\usepackage{t4phonet}
\usepackage{amssymb}
\usepackage{amsmath}
\usepackage{amsthm}
\usepackage{graphicx}
\usepackage{mathrsfs}
\usepackage{euscript}
\usepackage{mathtools}
\usepackage{bbm}
\usepackage{graphicx}
\usepackage{enumitem}
\usepackage{tikz}
\usepackage{tikz-cd}
\usepackage[all,cmtip]{xy}
\usetikzlibrary{decorations.markings,positioning}
\usepgflibrary{arrows}
\usepackage{setspace}
\usepackage{fancyhdr}
\usepackage{titlesec}
\usepackage[marginratio=1:1,top=3.2cm]{geometry}
\usepackage{url}
\usepackage{hyperref}
\usepackage{cleveref}
\newtheorem{mythmmain}{Theorem}

\newtheorem{myque}{Question}
\newtheorem{mythm}{Theorem}[section]
\newtheorem{mycor}[mythm]{Corollary}
\newtheorem{mylem}[mythm]{Lemma}
\newtheorem{mypro}[mythm]{Proposition}
\newtheorem{mycon}[mythm]{Conjecture}
\theoremstyle{definition}
\newtheorem{mydef}[mythm]{Definition}
\newtheorem{myex}[mythm]{Example}
\newtheorem{myrem}[mythm]{Remark}
\numberwithin{equation}{subsection}
\fancypagestyle{myheadings}{%
	\fancyhf{}
	\fancyhead[OC]{SHEAGAN A. K. A JOHN}
	\fancyhead[EC]{GENERALIZATIONS OF THE COINCIDENCE VALUE PROPERTY}
	\fancyfoot[C]{\thepage}%
	
}

\titleformat{\section}{\normalsize\centering\scshape}{\S\thesection\quad}{0pt}{}
\titleformat{\subsection}{\normalsize\centering\scshape}{\S\thesubsection\quad}{0pt}{}

\begin{document}
\title{\normalsize{GENERALIZATIONS OF THE COINCIDENCE VALUE PROPERTY}}
\author{Sheagan A. K. A. John}
\renewcommand\footnotemark{}
\thanks{The author was partially supported by NSF 1800737 and NSF 1952693}
\date{}
\maketitle	
\begin{abstract}
For $f,g:X\longrightarrow X$ continuous and commuting maps of a Hausdorff space, we investigate various conditions on $X$ and on the pair $(f,g)$ which provide existence of a coincidence value. We introduce generalized notions of the coincidence value property and use this added flexibility to determine how various coincidence properties of $X$ are related to group actions on $X$ and to coincidence properties of associated fibre bundles and adjunction spaces. We also present a sheaf theoretical approach to obtaining information concerning coincidence values through construction of an ``almost constant" presheaf. In particular, we prove several partial results concerning the special cases where $X$ is either a low dimensional dimensional sphere or the closed unit disk. 
\end{abstract}
\renewcommand\contentsname{Table of Contents}
\tableofcontents
\pagestyle{myheadings}	

\pagenumbering{arabic}	
\section{Introduction}

A topological space $X$ has the fixed point property (FPP) if for every continuous self-map $f:X\longrightarrow X$ there exists a point $x$ such that $f(x)=x$. Even where the FPP condition cannot be proven in full generality, determination of fixed points for classes of particularly important maps is a ubiquitous and well studied problem. Though the definition of the fixed point property is frameable entirely in the language of pure topology, the geometric, algebraic topological, and analytic properties of $X$ are often indispensable in proving non-trivial results. This is particularly true in the area of fixed point theory on manifolds, which frequently relies on differentiable structure or on extensions of the classic results of Nielsen and Lefschetz (e.g. \cite{SW09}, \cite{AW68},\cite{BCP99}). Many of the early and fundamental questions concerning the FPP for manifolds have been answered under hypotheses of various strength, particularly in the realm of Cartesian products of manifolds \cite{RB74,RB82}. However, there is still active and recent work-- such as the ``almost counterexample" of Kwasik and Sun \cite{KS17}-- on the longstanding unsolved problem of determining the FPP for $X\times X$ given $X$ is closed and possesses the FPP.

A topological space $X$ has the coincidence value property (CVP) if for every pair of commuting, continuous self-maps $f,g:X\longrightarrow X$ there exists a point $x$ such that $f(x)=g(x)$. This notion can be further strengthened by requiring the coincidence value to also be a fixed point; that is, $f(x)=g(x)=x$. A space which satisfies this latter condition is said to possess the common fixed point property (CFPP). It follows from the definitions that there is a chain of inclusions
\[\mathrm{CFPP}\subset\mathrm{CVP}\subset\mathrm{FPP}\]   
and moreover these containments are strict. In the mid 1950's the question as to whether or not every commuting pair of self-maps of the unit interval possessed a common fixed point was posed. Over a decade later it was independently proven by Huneke \cite{JH69} and Boyce \cite{WB69} that $X=[0,1]$ does not have the CFPP; see the recent survey article of Brown \cite{RB21} for a history of this problem. It is, however, a standard double application of the intermediate value theorem to prove that the unit interval is a CVP space. Huneke and Glover later provided in \cite{JHHG71} examples of spaces which do not have the CFPP, importantly proving this negative result for all manifolds with non-negative Euler characteristic. The problem of deciding when a pair of self-maps possess common fixed points, though more restrictive than determining coincidence values, is perhaps the more studied, with much work being done in the realm of mappings on metric spaces satisfying some ``contractibility" property (see Belluce and Kirk \cite{BK66}, Jungck \cite{GJ76,GJ05}, as well as \cite{AMR18}). 

Similarly to how the unit interval was pivotal in providing an example of a basic and non-pathological-- if originally unexpected-- topological space which distinguished between the CVP and CFPP, the natural next step might be to turn attention to investigating such a barometer space with respect to CVP and FPP. The following open problem is the inspiration for our investigations into coincidence value theorems and is attributed to Bellamy in \cite{WL83} but apparently has origins dating back to the 70's.
\begin{myque}\label{que1}
Do each two commuting functions on a simple triod have a common incidence point?
\end{myque}
Using the fact that the coincidence value property is preserved under retracts, coupled with the simple triod being a retract of the closed unit disk, \Cref{que1} is equivalent to asking whether the disk is a CVP space. This problem has been notoriously difficult to attack in the full generality of continuous self-maps on $\mathbf{D}$, whereas by contrast it was already proven in the positive by Shields \cite{AS64} (then extended by Behan \cite{DB73}) when restricting to commuting pairs of analytic self-maps. Later partial results similarly arise from considering classes of other suitably ``well-behaved" commuting pairs, such as contractible maps. We note that when one map is a homeomorphism, a positive result follows directly from the Brouwer fixed-point theorem. One key difficulty in approaching this problem through algebraic topological or geometrical means is the disk's lack of interesting structure. For example, all fibre bundles over the disk are trivial, and the triviality of all homotopy and singular homology/cohomology groups means that direct Lefschetz or Nielsen number type computations provide no interesting information. 

The above considerations suggest certain generalizations of the notion of the CVP, for which the aim is to introduce sufficient flexibility to admit analysis of coincidence values through non-trivial algebraic and geometric tools. That the standard proof of the Brouwer fixed-point theorem uses the existence of non-trivial homology/homotopy groups of the $n$-spheres $\mathbf{S}^n$ suggests that one approach to \Cref{que1} is to relate to self-maps of other spaces where such techniques are possible. To this end we introduce two generalized notions of the coincidence value property which extend to non-FPP spaces: \textit{weak coincidence value property} (WCVP) and \textit{periodic coincidence value property} (PCVP); the precise definitions are given in \Cref{defA1} and \Cref{defA2}, respectively. Our primary interest is in determining how conditions on $X$ influence these generalized coincidence value properties, though for important examples of topological spaces we follow general practice and investigate which extra restrictions placed on $f$ and $g$ will guarantee or deny a coincidence value. The first major contribution of this paper is in establishing-- as a special case of a much more general result proven in \Cref{secB}-- coincidence value relationships between the closed disk and the 2-sphere.
\begin{mythmmain}
The closed unit disk $\mathbf{D}$ has the CVP if and only if a coincidence value exists for every commuting pair $f,g:\mathbf{S}^2\longrightarrow\mathbf{S}^2$ such that $f$ is not surjective. The closed unit disk $\mathbf{D}$ has the PCVP if and only if there exists $x$ such that $f^n(x)=g^m(x)$ for every commuting pair $f,g:\mathbf{S}^2\longrightarrow\mathbf{S}^2$ where $f$ is not surjective.
\end{mythmmain}
Also in \Cref{secB}, we are able to prove ``lifting conditions" relating coincidence values of self-maps on a base space to existence of coincidence values for self-maps of an associated fibre bundle. Throughout this paper we have made a concerted effort to keep our results as topologically general and elementary as possible, thus the following theorem can almost certainly be improved by exploiting manifold structure or other common properties.
\begin{mythmmain}
Let $X$ be a CVP space and $E$ be the total space of a fibre bundle over $X$ having discrete, two-point fibre, then there exists $y\in E$ such that $f^n(y)=g^m(y)$ for $m,n\leq2$ and any arbitrary commuting bundle maps $f,g:E\longrightarrow E$. In the other direction, if every commuting pair of bundle maps (either both with fixed points or both fixed-point-free) $f,g:E\longrightarrow E$ have a coincidence value, then $X$ is a WCVP space. In particular, these relationships hold for the real projective plane and its double cover $\mathbf{S}^2$.
\end{mythmmain}
It is so far not known which topological conditions guarantee that $X$ has the CVP; we know, for example, that the FPP is not preserved with respect to homotopy equivalence even for $X$ a compact and contractible Hausdorff space \cite{SK53}. It follows that unless strong restrictions are placed on $X$ or the collection of self-maps under investigation any algebraic topological approach to determining existence of coincidence values needs to extract much finer information than mere global homotopy invariants. In practice, some combination of restriction is generally required (e.g. \cite{BB-M05},\cite{MH89},\cite{PS01}), where increased generality in one aspect leads to imposition of increasingly subtle hypotheses in the other direction. Indeed, one of the key attributes of the following main result of \Cref{secB'} is that no explicit restrictions are placed on the maps, this being subsumed into the conditions placed on $X$.
\begin{mythmmain}
Let $X$ be a paracompact Hausdorff space such that if $f,g:X\longrightarrow X$ commute and $f$ is a homeomorphism, then there exists $x$ satisfying $f^n(x)=g^m(x)$. There exists a sheaf of Abelian groups such that $X$ possesses the PCVP if and only if a certain submodule of the collection of global sections only contains elements having torsion image. Conversely, $X$ is not a PCVP space if and only if there exists an element of this submodule which has image of infinite order. In addition, if $X$ admits an admissible cover $\mathscr{U}=\{U_1,U_2\}$ (see \Cref{defBd2}) of open sets, then $X$ is a PCVP space if both subspaces are.
\end{mythmmain}
This paper is organized as follows. In \Cref{secA} we provide some of the basic tools which are used throughout the rest of the paper; particularly, we show that the generalized coincidence value properties are topological invariants. The results of \Cref{secB} and \Cref{secB'} are separated into two distinct approaches, both of which are inspired by the coincidence value properties of the $n$-spheres. The geometric/topological flavoured argumentation of \Cref{secB1} provides connections between coincidence values of a space $X$ and the fibre bundles and adjunction spaces associated to it; in \Cref{secB2} we expand on some questions raised in \Cref{secB1}, primarily by investigating how existence of coincidence values for pairs of fixed-point-free maps is determined by the presence of various group actions. On the other hand, \Cref{secB3,secB4} are decidedly algebraic topological in scope, with \Cref{secB3} introducing a sheaf theoretical framework for determining coincidence values on suitably nice spaces, and \Cref{secB4} is given over to investigating basic cohomological properties of this sheaf, such as the existence of Mayer-Vietoris sequences and a K\"{u}nneth theorem. 

\section{Preliminaries}\label{secA}
Given a continuous self-map $f:X\longrightarrow X$, we shall denote its collection of fixed points by $\mathcal{FP}(f)$; note that since $X$ will always be assumed to be Hausdorff, $\mathcal{FP}(f)$ is closed and will thus be (locally) compact if $X$ is (locally) compact. It is easy to see that if $f,g:X\longrightarrow X$ are commuting maps then $g(\mathcal{FP}(f))\subseteq\mathcal{FP}(f)$ and $f(\mathcal{FP}(g))\subseteq\mathcal{FP}(g)$. Since we shall have occasion to refer to it often, we shall give a proof of the following basic result.

\begin{mypro}\label{proA1}
If $\mathcal{FP}(fg)\neq\emptyset$ for commuting self-maps $f,g:X\longrightarrow X$, then $f$ and $g$ are bijections on $\mathcal{FP}(fg)$.
\end{mypro}
\begin{proof}
Suppose that $x\in\mathcal{FP}(fg)$. Then $fg(g(x))=g(fg(x))=g(x)$ and $fg(f(x))=f(gf(x))=f(x)$, so $f(x),g(x)\in\mathcal{FP}(fg)$. To prove surjectivity of $f$ note that $x=f(g(x))$ for any $x\in\mathcal{FP}(fg)$; likewise we have surjectivity of $g$. If $f$ is not injectiive then $\exists y\neq x\in\mathcal{FP}(fg)$ such that $f(x)=f(y)$. But then $x=gf(x)=gf(y)=y$, which is a clear contradiction.
\end{proof}
In the above context it should also be noted that $f$ and $g$ are inverses to one another, a property which is easily checked by realizing that both $fg$ and $g^{-1}f^{-1}$ are the identity on $\mathcal{FP}(fg)$. Indeed, this provides us with a similarly basic implication concerning the closure of $\mathcal{FP}(fg)$ under preimages.

\begin{mypro}\label{proA2}
If $f(y)\in\mathcal{FP}(fg)\neq\emptyset$ for commuting self-maps $f,g:X\longrightarrow X$, then $y\in\mathcal{FP}(fg)$; likewise for $g(y)\in\mathcal{FP}(fg)$. 
\end{mypro}
\begin{proof}
Suppose $f(y)\in\mathcal{FP}(fg)$ for some $y\in X$. Then $fgf(y)=f(y)$. However, from \Cref{proA1} we know that $f$ is a bijection on $\mathcal{FP}(fg)$ where $f^{-1}_{|\mathcal{FP}(fg)}\equiv g$, from which it follows that $fg(y)=gf(y)=f^{-1}_{|\mathcal{FP}(fg)}f(y)=y$, and so $y$ is a fixed point of $fg$.
\end{proof}

For each pair $f,g:X\longrightarrow X$ of self-maps, the set of coincidence values $\mathcal{CV}(f,g)$ satisfies weaker versions of the above propositions. We will simply state the necessary results, seeing as the proofs are essentially identical; again we note that if $X$ is Hausdorff then $\mathcal{CV}(f,g)$ is closed, being compact if $X$ is itself compact.

\begin{mypro}\label{proA3}
For any pair of not necessarily commuting self-maps $f_1,f_2:X\longrightarrow X$, if $g$ is a self-map which commutes with both $f_1$ and $f_2$, then $g(\mathcal{CV}(f_1,f_2))\subseteq\mathcal{CV}(f_1,f_2)$. In particular, if $\mathcal{CV}(f,g)\neq\emptyset$ for commuting self-maps $f,g:X\longrightarrow X$ where $f$ is a bijection, then $g$ is a bijection on $\mathcal{CV}(f,g)$, and if  either $f(y)$ or $g(y)$ belong to $\mathcal{CV}(f,g)$ then $y\in\mathcal{CV}(f,g)$.
\end{mypro}

The following important lemma can be found in basically equivalent form in \cite[Proposition 4.1]{GJ88}, though we refer the reader to Dowell's survey \cite[Lemma 4]{EM09} for a particularly nice elementary proof. 

\begin{mylem}\label{lemA1}
Let $X$ be a compact and second-countable Hausdorff space. For commuting self-maps $f,g:X\longrightarrow X$ there exists a nonempty compact subspace $K\subseteq X$ such that $f(K)=g(K)=K$. Moreover, if $X$ is connected then $K$ can be taken to be connected as well.
\end{mylem}
\begin{mycor}\label{corA1}
Under the conditions of \Cref{lemA1} the subspace $K$ is contained in the compact (and connected) subspace $f(X)\cap g(X)$ and thus contains $\mathcal{CV}(f,g)$.
\end{mycor}
\begin{proof}
By the above lemma there exists non-empty compact (connected) $K\subseteq X$ such that $f(K)\cap g(K)=K$; thus, $K\subseteq f(X)\cap g(X)\neq\emptyset$. Continuity of $f$ and $g$ guarantee compactness (connectedness) of $f(X)\cap g(X)$ and it follows by definition that $\mathcal{CV}(f,g)$ is a subset.
\end{proof}

The following definitions introduce weakened notions of the coincidence value property such that there is no longer any dependence upon every self-map possessing a fixed point. However, we still have that the collection of CVP spaces is contained within these weaker collections. 
\begin{mydef}\label{defA1}
A topological space $X$ has the weak coincidence value property (WCVP) if  there exists $x\in\mathcal{CV}(f,g)$ for every pair of commuting self-maps $f,g:X\longrightarrow X$, given that both are fixed-point-free or both have fixed points.
\end{mydef}
\begin{mydef}\label{defA2}
A topological space $X$ has the periodic coincidence value property (PCVP) if for every pair of commuting self-maps $f,g:X\longrightarrow X$ there exists a point $x$ such that $x\in\mathcal{CV}(f^n,g^m)$ for $m,n\in\mathbb{N}\setminus\{0\}$. A stronger condition is that of $X$ being a uniformly $k$-PCVP space; that is, there exists a single $k\in\mathbb{N}$ such that $k\geq\max\{m,n\}$ for every pair $(m,n)$.
\end{mydef}
\begin{mythm}\label{thmA1}
The WCVP and (uniform) PCVP are topological invariants
\end{mythm}
\begin{proof}
Let $X$ be a WCVP space and $Y$ a space homeomorphic to $X$, where $\phi:X\longrightarrow Y$ is a homeomorphism. That $Y$ is also a WCVP space follows from the fact there exists a unique commuting pair of self-maps
\[\phi^{-1}f\phi:X\longrightarrow X\qquad\mathrm{and}\qquad\phi^{-1}g\phi:X\longrightarrow X\]
for every commuting pair of self-maps $f,g:Y\longrightarrow Y$; hence $\phi(x)\in\mathcal{CV}(f,g)$. Similarly, if $X$ is a PCVP space with $f,g:Y\longrightarrow Y$ a commuting pair of self-maps, then it is straightforward to verify that $x\in\mathcal{CV}((\phi^{-1}f\phi)^n,(\phi^{-1}g\phi)^m)=\mathcal{CV}(\phi^{-1}f^n\phi,\phi^{-1}g^m\phi)$ for some $m,n\in\mathbb{N}\setminus\{0\}$. This implies that $\phi(x)\in\mathcal{CV}(f^n,g^m)$.
\end{proof}
Below we provide examples of compact Hausdorff WCVP and PCVP spaces which do not satisfy the CVP, hence these definitions are not vacuous. We also note that they are not entirely weak either, preserving the notion that certain maps are too special to possess coincidence values. For example, the circle $\mathbf{S}^1$ has neither the WCVP nor the PCVP since every pair of irrational rotations $\theta_1,\theta_2$ commute and are fixed-point-free but have no coincidence value; moreover any $n$-fold composition of an irrational rotation is again an irrational rotation, so $\mathcal{CV}(f^n,g^m)=\emptyset$.
\begin{myex}\label{exaA1}
Consider $X=[0,1]\sqcup[2,3]$ with the topology inherited as a subspace of $\mathbb{R}$. Since connectedness is preserved under continuity, any self map $f:X\longrightarrow X$ has one of the following behaviours
\[\mathrm{(i)}\;f:[0,1]\longrightarrow [0,1]\quad\mathrm{and}\quad f:[2,3]\longrightarrow [2,3]\qquad \mathrm{(ii)}\;f:X\longrightarrow[0,1]\]	 
\[\mathrm{(iii)}\;f:[0,1]\longrightarrow [2,3]\quad\mathrm{and}\quad f:[2,3]\longrightarrow [0,1]\qquad \mathrm{(iv)}\;f:X\longrightarrow[2,3]\]
Suppose that $f,g:X\longrightarrow X$ commute and are of some type (i),(ii) or (iv); by considering restriction to one interval we have that $f,g:[0,1]\longrightarrow[0,1]$ or $f,g:[2,3]\longrightarrow[2,3]$. Since every closed and bounded interval has the CVP it follows that $f$ and $g$ possess a coincidence value. It is obvious that $X$ itself is not a CVP space since every type (iii) self-map is fixed-point-free. Now suppose that $f$ and $g$ commute but $f$ is of type (iii), then $g$ must be either type (i) or (iii). If $g$ is of type (i) it follows that $f^2,g:[0,1]\longrightarrow[0,1]$ are commuting self-maps and thus $\mathcal{CV}(f^2,g)$ is nonempty. In the other case we similarly see that $f^2,g^2:[0,1]\longrightarrow[0,1]$ are commuting self-maps and thus $\mathcal{CV}(f^2,g^2)$ is nonempty.
\end{myex} 
\begin{myex}\label{exaA2}
Consider the disjoint union $X=[0,1]\sqcup\{2,3\}$ where the unit interval has the usual topology and $\{2,3\}$ has the trivial topology. The set $\{2,3\}$ must be open under the preimage of any continuous self-map, hence $f^{-1}(\{2,3\})$ must either be $X,\{2,3\}$, or contained entirely within $[0,1]$. Moreover, since connectedness is preserved under continuity any self-map $f:X\longrightarrow X$ has one of the following behaviours
\[\mathrm{(i)}\;f:[0,1]\longrightarrow [0,1]\quad\mathrm{and}\quad f:\{2,3\}\longrightarrow \{2,3\}\qquad \mathrm{(ii)}\;f:X\longrightarrow[0,1]\]	 
\[\mathrm{(iii)}\;f:[0,1]\longrightarrow \{2,3\}\quad\mathrm{and}\quad f:\{2,3\}\longrightarrow [0,1]\qquad \mathrm{(iv)}\;f:X\longrightarrow\{2,3\}\]
Type (iv) maps only commute with each other and possess a fixed point if and only if they are constant or the identity on $\{2,3\}$; clearly any such commuting $f$ and $g$ agree on $\{2,3\}$. Now suppose that $f,g:X\longrightarrow X$ commute, have a fixed point, and they are of type (i) or (ii). Then by considering restriction to the interval we have that $f,g:[0,1]\longrightarrow[0,1]$ so as in the previous example $f$ and $g$ possess a coincidence value. It is obvious that a self-map is fixed-point-free if and only if it is of type (iii) or a type (iv) which permutes $\{2,3\}$; moreover, a pair of such maps commute only if they are of the same type. However, it is easy to check that fixed-point-free type (iv) maps commute only if they agree on $\{2,3\}$-- and indeed on the whole of $X$. Thus suppose that $f$ and $g$ are commuting type (iii) self-maps and 
note that $f^2,fg:[0,1]\longrightarrow[0,1]$ are commuting self-maps and thus possess a coincidence point. Without loss of generality let $0\in\mathcal{CV}(f^2,fg)$ and set $f(0)=2$, then: $g(2)=gf(0)=fg(0)=f^2(0)=f(2)$.
To emphasize that there truly do exist distinct commuting self-maps consider the type (iii) pair given by
\[g(2)=g(3)=f(2)=1,f(3)=1/2\quad\mathrm{and}\quad f([0,1])=g([0,1])=2\]
\end{myex} 

When they exist, a great deal of information about $X$ can be determined by examining the extra structure introduced by group actions. In all that follows we will assume $\Gamma$ to be a non-trivial group with identity denoted $\mathbbm{1}$ and $\gamma$ an arbitrary element. 
\begin{mydef}\label{defA3}
A topological space is a (left) $\Gamma$-space if it admits a continuous (left) $\Gamma$-action. That is, the bilinear map $\Gamma\times X\longrightarrow X$ given by $(\gamma,x)\rightarrow\gamma\cdot x$ is continuous with respect to the product topology of $\Gamma\times X$ and satisfies
\[\mathbbm{1}\cdot x=x\quad\forall x\in X\qquad(\gamma_1\gamma_2)\cdot x=\gamma_1\cdot(\gamma_2\cdot x)\quad\forall\gamma_1,\gamma_2\in\Gamma\;,\forall x\in X\]
A continuous map $f:X\longrightarrow Y$ between two (left) $\Gamma$-spaces is called $\Gamma$-equivariant if $f(\gamma\cdot x)=\gamma\cdot f(x)$ for all $\gamma\in\Gamma$ and $x\in X$.
\end{mydef}
A $\Gamma$-action on $X$ is said to be free if $\gamma\cdot x=x$ implies that $\gamma=\mathbbm{1}$; a group action is transitive if for any pair $x,y\in X$ there exists $\gamma\in\Gamma$ such that $\gamma\cdot x=y$. An action is regular if it is both free and transitive, the left multiplication of a topological group on itself being such an example. There are weaker notions of these pointwise actions; for example, $X$ admits a \textit{topologically transitive} action if for every pair of open sets $U,V\subseteq X$ there exists $\gamma$ such that $(\gamma\cdot U)\cap V\neq\emptyset$.

\begin{mydef}\label{defA4}
A locally compact space $X$ admits a proper $\Gamma$-action if the map from $\Gamma\times X$ to $X\times X$ defined by $(\gamma,x)\mapsto(\gamma\cdot x,x)$ is a proper map. Viewing $\Gamma$ as a topological group, this is equivalent to (see \cite[Chapter 21]{JML13}) the following characterization:
\[\forall\;\mathrm{compact}\;K\subsetneq X:\quad\Gamma_K=\{\gamma\in\Gamma:\;(\gamma\cdot K)\cap K\neq\emptyset\}\;\mathrm{is\;compact}\]
\end{mydef}

\begin{mydef}\label{defA5}
For a topological space $F$ a fibre bundle with model fibre $F$ over a (locally) compact Hausdorff space $X$ is a triple $(E,X,\pi)$ where $\pi:E\longrightarrow X$ is a surjective continuous map. Every arbitrary fibre bundle $(E,X,\pi)$ shall be taken to be locally trivial; that is there exists an open cover $\{U_\alpha\}$ of $X$ and a family of homeomorphisms $\{\phi_\alpha\}$ such that
\[\phi_\alpha:\pi^{-1}(U_\alpha)\longrightarrow U_\alpha\times F\qquad\mathrm{and}\qquad\phi_\alpha:\pi^{-1}(\{x\})\longrightarrow \{x\}\times F\quad\forall x\in U_\alpha\]
The collection $\{U_\alpha,\phi_\alpha\}$ is called a set of local trivializations of the bundle.
\end{mydef}

\begin{mydef}\label{defA6}
If $F$ is a left $\Gamma$-space for some topological group $\Gamma$, then a fibre bundle with model fibre $F$ is said to have \textbf{\textit{structure group}} $\Gamma$ if the patching together of overlapping local homeomorphisms is governed by $\Gamma$. Explicitly, there exists a collection of continuous transition functions satisfying
\[\rho_{\alpha_i\alpha_j}:U_{\alpha_i}\cap U_{\alpha_j}\longrightarrow\Gamma\leq\mathsf{Homeo}(F)\qquad\rho_{\alpha_i\alpha_j}=\mathbbm{1}\;if\;i=j\qquad\rho_{\alpha_i\alpha_k}=\rho_{\alpha_i\alpha_j}\circ\rho_{\alpha_j\alpha_k}\]
\[\phi_{\alpha_i}\phi_{\alpha_j}^{-1}:(U_{\alpha_i}\cap U_{\alpha_j})\times F\longrightarrow (U_{\alpha_i}\cap U_{\alpha_j})\times F\qquad(x,y)\longmapsto(x,\rho_{\alpha_i\alpha_j}(x)\cdot y)\]
In the special case that the $\Gamma$-action on $F$ is regular, the bundle $(E,X,\pi)$ is called a principal $\Gamma$-bundle. The fibre bundle $(\mathcal{E},X,p)$ where $\mathcal{E}=X\times F$ and $p:\mathcal{E}\longrightarrow X$ is the obvious projection map will be called the trivial fibre bundle. 
\end{mydef}

\begin{mydef}\label{defA7}
For any continuous map $f:Y\longrightarrow X$ and arbitrary fibre bundle $(E,X,\pi)$ the \textbf{\textit{pullback bundle}} $(f^*E,Y,\pi_f)$ is defined by
\[f^*E=\{(y,e)\in Y\times E:\;f(y)=\pi(e)\}\quad\mathrm{and}\quad\pi_f(y,e)=y\]
\end{mydef}

We say that there exists a morphism between fibre bundles $(E_1,X,\pi)$ and $(E_2,Y,\sigma)$ if there exist a pair of maps $(\varphi,h)$ such that the following diagram commutes
\[\begin{tikzcd} 
E_2 \arrow{r}{\varphi} \arrow{d}{\sigma} & E_1 \arrow{d}{\pi}\\
Y \arrow{r}{h} & X
\end{tikzcd}\]
and $\varphi$ is a continuous map on each fibre. In this light it is often useful to view the pullback bundle $f^*E$ as the universal-- in the sense of category theory-- pullback over $X$ where the morphism between $f^*E$ and $E$ is given by $h=f$ and $\varphi$ is the projection onto the second factor. 

\begin{mydef}\label{defA8}
Two fibre bundles $(E_1,X,\sigma_1)$ and $(E_2,X,\sigma_2)$ are isomorphic if the bundle map $\varphi$ is a continuous bijection on fibres. 
\end{mydef}
When the context is obvious will write $E$ to denote the fibre bundle triple and $E_1\cong E_2$ to reflect isomorphic bundles.

\begin{mydef}\label{defA9}
Let $X$ be a topological space and $\mathbf{C}$ a category (such as the category $\mathbf{Mod}_{\mathbb{Z}}$ of $\mathbb{Z}$-modules) with objects $\mathsf{Obj}(\mathbf{C})$ and morphisms $\mathsf{Mor}(\mathbf{C})$. Let $W\subseteq V\subseteq U\subseteq X$ be arbitrary open sets with inclusion maps $j:W\hookrightarrow V$ and $i:V\hookrightarrow U$. A presheaf $\mathcal{F}$ on (the category of open sets) $X$ which takes values in $\mathbf{C}$ is a contravariant functor satisfying
\[\mathcal{F}(i\circ j)=\mathcal{F}(j)\circ\mathcal{F}(i)\]
\[\mathcal{F}(Id_U)=\mathbbm{1}_{\mathcal{F}(U)}\]
where $\mathcal{F}(U)\in\mathsf{Obj}(\mathbf{C})$ and the map $\mathcal{F}(i):\mathcal{F}(U)\longrightarrow\mathcal{F}(V)$ belongs to $\mathsf{Mor}(\mathbf{C})$. To make immediate the relationship between the open sets we will usually denote the morphism $\mathcal{F}(i)$ by $\rho^U_V$.
\end{mydef}

Given presheaves $\mathcal{F}$ and $\mathcal{G}$ on $X$, $\varphi:\mathcal{F}\longrightarrow\mathcal{G}$ is called a morphism of presheaves if there exists a family of maps $\varphi_U\in\mathsf{Mor}(\mathbf{C})$ such that the following diagram commutes for every inclusion of open sets $V\subseteq U\subseteq X$.
\[\begin{tikzcd} 
\mathcal{F}(U) \arrow{r}{\varphi_U} \arrow[d,swap,"(\rho_\mathcal{F})^U_V"] & \mathcal{G}(U) \arrow{d}{(\rho_\mathcal{G})^U_V}\\
\mathcal{F}(V) \arrow{r}{\varphi_V} & \mathcal{G}(V)
\end{tikzcd}\]

\begin{mydef}\label{defA10}
A monopresheaf on $X$ which takes values in $\mathbf{C}$ is a presheaf which satisfies the additional condition that for any open set $U\subseteq X$ with arbitrary open cover $\{U_\alpha\}_{\alpha\in\Lambda}$, if $\rho^U_{U_\alpha}(a)=\rho^U_{U_\alpha}(b)$ for some $a,b\in\mathcal{F}(U)$ and all $\alpha\in\Lambda$, then $a=b$. A sheaf on $X$ which takes values in $\mathbf{C}$ is a monopresheaf which satisfies the gluing condition 
\[\rho^{U_\alpha}_{U_\alpha\cap U_\beta}(a_\alpha)=\rho^{U_\beta}_{U_\alpha\cap U_\beta}(a_\beta)\;\mathrm{for\;all}\;\alpha,\beta\in\Lambda\implies\exists a\in\mathcal{F}(U)\;\mathrm{such\;that}\;\rho^U_{U_\alpha}(a)=a_\alpha\]
with respect to any arbitrary open cover $\{U_\alpha\}_{\alpha\in\Lambda}$ of $U$ and for every collection of elements $\{a_\alpha\}_{\alpha\in\Lambda}$ with $a_\alpha\in\mathcal{F}(U_\alpha)$.

\end{mydef}

\section{Realizing WCVP and PCVP Structure}\label{secB}
Elementary degree theory and the examples of \Cref{secA} would suggest that for ``nice enough" spaces there is generally a sharp partition of behaviour among the classes of fixed-point-free, fixed-point-possessing, surjective, and non-surjective maps. The special case of self-homeomorphisms is particularly distinguishing and warrants the following definition.
\begin{mydef}\label{defB1}
A space $X$ is $\mathsf{Homeo}$-WCVP (respectively $\mathsf{Homeo}$-PCVP) if it satisfies the WCVP (respectively PCVP) conditions given that one of the maps is a homeomorphism.
\end{mydef}
Before investigating the coincidence value properties in generality it is worthwhile to gain intuition by examining the case of spheres, using the already well developed geometric and algebraic tools available to us.
\begin{mypro}\label{proB1}
Every even dimensional sphere is both a $\mathsf{Homeo}$-WCVP and a uniformly $\mathsf{Homeo}$-2-PCVP space.
\end{mypro}
\begin{proof}
Let $f,g:\mathbf{S}^{2n}\longrightarrow\mathbf{S}^{2n}$ be commuting maps such that $f$ (hence also $f^{-1}$) is a homeomorphism. By the even dimensional assumption we know from basic degree theory that a self-map of $\mathbf{S}^{2n}$ is fixed-point-free if and only if its degree is equal to $-1$. Moreover, since degree is multiplicative it is easy to see that
\[1=\mathsf{deg}(Id_{\mathbf{S}^{2n}})=\mathsf{deg}(f)\cdot\mathsf{deg}(f^{-1})\implies\mathsf{deg}(f)=\mathsf{deg}(f^{-1})\in\{1,-1\}\]
Concerning the $\mathsf{Homeo}$-WCVP, if both $f$ and $g$ are fixed-point-free, then $\mathsf{deg}(f^{-1}g)=(-1)\cdot(-1)=1$ and thus $\mathcal{FP}(f^{-1}g)=\mathcal{CV}(f,g)\neq\emptyset$. If both $f$ and $g$ possess fixed points, then $\mathsf{deg}(f^{-1}g)=(-1)\cdot(-1)=1$ and thus $\mathcal{FP}(f^{-1}g)=\mathcal{CV}(f,g)\neq\emptyset$. Following similar argumentation we see that $\mathsf{deg}(f^{-2}g^2)=\mathsf{deg}(f^{-1}g)^2\neq-1$ for any commuting pair $(f,g)$ and so  $\mathcal{FP}(f^{-2}g^2)=\mathcal{CV}(f^2,g^2)\neq\emptyset$, which proves the $\mathsf{Homeo}$-2-PCVP assertion.
\end{proof}

The following result is a direct consequence of the work of Hoffman on the noncoincidence index for manifolds (see \cite{MH84} for a definition), and the original proof can be found as an application of \cite[Proposition 3.3]{MH84}.  
\begin{mypro}\label{proB2}
Every pair of fixed-point-free commuting maps $f,g:\mathbf{S}^{2n}\longrightarrow\mathbf{S}^{2n}$ possesses a coincidence value.
\end{mypro}
\begin{proof}
First note that $\mathbf{S}^{2n}$ is an oriented compact manifold with Euler characteristic $\chi(\mathbf{S}^{2n})=2\neq0$ and Betti numbers $\beta_1=\beta_2=1$; moreover we know that every degree zero self-map must possess a fixed point. It follows that $\mathbf{S}^{2n}$ satisfies the conditions of \cite[Theorem 1.1]{MH89}, and thus the noncoincidence index of $\mathbf{S}^{2n}$ is less than or equal to $\beta_1^2+\beta_2^2=2$. Hence any collection $\{f_1,f_2,\ldots,f_k\}$ of two or more fixed-point-free self-maps must possess a distinct pair $(f_i,f_j)$ such that $\mathcal{CV}(f_i,f_j)\neq\emptyset$.
\end{proof}

The existence of irrational rotations shows that such results are false for the circle. In fact, since $\chi(\mathbf{S}^{2n+1})=0$ by \cite[Theorem 2.1]{MH84} the noncoincidence index of every odd dimensional sphere must be infinite. Thus there exists a collection  $\{f_1,f_2,\ldots\}$ of fixed-point-free self-maps $f_i:\mathbf{S}^{2n+1}\longrightarrow\mathbf{S}^{2n+1}$ such that $\mathcal{CV}(f_i,f_j)=\emptyset$ for every distinct pair $(f_i,f_j)$\footnote{Note that this is technically not strong enough to disprove WCVP structure since we cannot assume that any pair of maps commute.}. However, it is easy to prove by explicit and elementary argument that imposition of non-surjectivity recovers the existence of coincidence points for the circle.
\begin{myex}\label{exaB1}
Every pair of commuting maps $f,g:\mathbf{S}^1\longrightarrow\mathbf{S}^1$ possesses a coincidence value given that one map is not surjective.
\end{myex}
\begin{proof}
Since the circle is a compact and connected Hausdorff space by \Cref{lemA1} there exists a compact and connected subspace $K\subseteq\mathbf{S}^1$ such that $f(K)=g(K)=K$. If we suppose that $g$ is not surjective, then $K$ is a strict subspace of $\mathbf{S}^1$. However, every connected subspace of the circle is path connected and so the only such compact proper subspace is a closed arc. Every closed arc is homeomorphic to the unit interval and thus has the FPP, since the FPP is a topological invariant. Given that $x_1=e^{2\pi i\theta_1}$ and $x_2=e^{2\pi i\theta_2}$ are the endpoints of $K$ then
\[\gamma(t):=e^{2\pi i(t\theta_2+(1-t)\theta_1)}\qquad t\in[0,1]\]
describes a path connecting the endpoints. Furthermore we can order $K$ according to $\gamma(t_x)=x\geq y=\gamma(t_y)$ if $t_x\geq t_y$. Hence $f(x)=g(x)$ if and only if $\gamma(t_{f(x)})=\gamma(t_{f(x)})$ if and only if $t_{f(x)}=t_{g(x)}$. Suppose without loss of generality that $\gamma(t_{f(x)})=f(x)>g(x)=\gamma(t_{g(x)})$ for all $x$-- otherwise by definition of $\gamma$ there must exist $x_0\in K$ such that $t_{f(x_0)}=t_{g(x_0)}$. The set $\mathcal{FP}(f)\cap K$ is itself compact and thus there exists a maximal $z\in\mathcal{FP}(f)\cap K$. By the commuting property we have that $g(z)\in\mathcal{FP}(f)\cap K$, which leads to the contradiction of the maximality of $z$, namely $z=f(z)<g(z)=f(g(z))$
\end{proof}

The fact that CVP is equivalent to WCVP+FPP but yet is still generally more restrictive than PCVP+FPP might imply that the WCVP is a stronger notion than PCVP. We dot not attempt to investigate any such claim in this direction, however there is an easy implication in the case of $\mathbf{S}^{2n}$.
\begin{mylem}\label{lemB1}
If an even dimensional sphere possesses the WCVP then it is also a PCVP space.
\end{mylem}
\begin{proof}
Let $f,g:\mathbf{S}^{2n}\longrightarrow\mathbf{S}^{2n}$ be commuting maps. Then by assumption $\mathcal{CV}(f,g)\neq\emptyset$ if either both maps are fixed-point-free or both possess fixed points. Without loss of generality we may assume that $f$ is fixed-point-free but $\mathcal{FP}(g)\neq\emptyset$; it then follows from elementary homotopy theory that $\mathsf{deg}(f)=-1$ and $\mathsf{deg}(g)\neq-1$. Since $\mathsf{deg}(f^2)=\mathsf{deg}(f)\cdot\mathsf{deg}(f)=1\neq-1$ it follows that $f^2$ possesses a fixed point and so by the WCVP hypothesis we know that $\mathcal{CV}(f^2,g)\neq\emptyset$.
\end{proof}

\subsection{Principal Bundles and Adjunction Spaces}\label{secB1}
\begin{mydef}\label{defBa1}
Let $(E,X,\pi)$ be a fibre bundle over a locally compact Hausdorff space $X$ and $\varphi_1,\varphi_2:E\longrightarrow E$ any pair of fibre-wise commuting self-morphisms. The bundle has the fibred-WCVP if  there exists a fibre $F_x=\pi^{-1}(\{x\})$ such that $\pi(\varphi_1(F_x))=\pi(\varphi_2(F_x))$, given that each morphism either maps some fibre into itself or both do not. The collection of all such fibres satisfying $\pi(\varphi_1(F_x))=\pi(\varphi_2(F_x))$ is denoted $\mathcal{CV}(\varphi_1,\varphi_2)_{\mathrm{fibre}}$.
\end{mydef}
Note that the condition $\pi(\varphi_1(F_x))=\pi(\varphi_2(F_x))$ is equivalent to $\varphi_1(F_x)$ and $\varphi_2(F_x)$ being subsets of some common fibre, and for convenience we shall sometimes use the notation $\varphi_1(F_x)\asymp_\pi\varphi_2(F_x)$ to express this relation.

\begin{mylem}\label{lemBa1}
A locally compact Hausdorff space $X$ is a WCVP space if and only if every fibre bundle $(E,X,\pi)$ has the \textit{fibred}-WCVP.
\end{mylem}
\begin{proof} Let $(E,X,\pi)$ be a fibre bundle with $(\varphi_1,h_1)$ and $(\varphi_2,h_2)$ self-morphisms of the bundle. By the definition of local triviality we always have the following diagram of two commuting squares where $p:X\times F\longrightarrow X$ is the usual projection map and  $\{U_\alpha,\phi_\alpha\}$ a family of local trivializations.
\[\begin{tikzcd} 
\pi^{-1}(U_\alpha) \arrow[r,"\varphi_i"] \arrow[d,"\pi"] & \pi^{-1}(U_{\beta_i}) \arrow[d,"\pi"] \arrow[r,"\phi_{\beta_i}"] & U_{\beta_i}\times F \arrow[d,"p"] \\
U_\alpha \arrow[r,"h_i"] & U_{\beta_i}  \arrow[r,"Id_X"] & U_{\beta_i}
\end{tikzcd}\]
If $\varphi_1,\varphi_2:E\longrightarrow E$ commute fibre-wise then it immediately follows from the left-hand square that $h_1,h_2:X\longrightarrow X$ are commuting self-maps.
\[h_1h_2=\pi\varphi_1\pi^{-1}\circ\pi\varphi_2\pi^{-1}=\pi\varphi_1\varphi_2\pi^{-1}=\pi\varphi_2\varphi_1\pi^{-1}=\pi\varphi_2\pi^{-1}\circ\pi\varphi_1\pi^{-1}=h_2h_1\]
It is also straightforward to verify that $\varphi_i$ fixes a fibre $F_x=\pi^{-1}(\{x\})$ if and only if $x$ is a fixed point of $h_i$. Assuming that that $X$ is a WCVP space, there exists $z\in\mathcal{CV}(h_1,h_2)$ from which it follows that $h_1(z)=h_2(z)\in U_{\beta_1}\cap U_{\beta_2}$; moreover, by the right-hand square $h_i\equiv Id_X^{-1}p\phi_{\beta_i}\varphi_i\pi^{-1}$ so $z\in\mathcal{CV}(p\phi_{\beta_1}\varphi_1\pi^{-1},p\phi_{\beta_2}\varphi_2\pi^{-1})$ and hence the choice of continuous maps $h_1,h_2$ does not matter. It is clear from the definition of $p$ that $z\in\mathcal{CV}(p\phi_{\beta_1}\varphi_1\pi^{-1},p\phi_{\beta_2}\varphi_2\pi^{-1})$ if and only if $\phi_{\beta_1}\varphi_1(F_z)$ is contained within the same fibre as $\phi_{\beta_2}\varphi_2(F_z)$. Application of the atlas/chart lemma for fibre bundles provides existence of a transition function $\rho_{\beta_1\beta_2}:U_{\beta_1}\cap U_{\beta_2}\longrightarrow\mathsf{Homeo}(F)$ such that $\phi_{\beta_1}\phi_{\beta_2}^{-1}(x,y)=(x,\rho_{\beta_1\beta_2}(x)\cdot y)$ for every $(x,y)\in U_{\beta_1}\cap U_{\beta_2}\times F$. Since $\phi_{\beta_i}:F_x\longrightarrow\{x\}\times F$ is an isomorphism, considering the action over a whole fibre
\begin{equation}
\begin{gathered}
\phi_{\beta_1}\phi_{\beta_2}^{-1}(\{x\}\times F)=\bigcup_{y\in F}\phi_{\beta_1}\phi_{\beta_2}^{-1}(x,y)=\bigcup_{y\in F}(x,\rho_{\beta_1\beta_2}(x)\cdot y)\\
=\bigcup_{\rho_{\beta_1\beta_2}(x)^{-1}\cdot y\in F}(x,y)=\bigcup_{y'\in F}(x,y)=\{x\}\times F
\end{gathered}
\end{equation}
where the second to last equality stems from $\rho_{\beta_1\beta_2}(x)$ belonging to $\mathsf{Homeo}(F)$. We thus obtain an equivalence as fibre-wise maps $\phi_{\beta_1}\equiv_{\mathrm{fibre}}\phi_{\beta_2}$ for all $x\in U_{\beta_1}\cap U_{\beta_2}$, and it follows that $(E,X,\pi)$ has the fibred-WCVP
\[\phi_{\beta_1}\varphi_1(F_z)\asymp_p\phi_{\beta_2}\varphi_2(F_z)\implies\varphi_1(F_z)\asymp_\pi\phi_{\beta_1}^{-1}\phi_{\beta_2}\varphi_2(F_z)\equiv_{\mathrm{fibre}}\phi_{\beta_2}^{-1}\phi_{\beta_2}\varphi_2(F_z)\asymp_\pi\varphi_2(F_z)\]
Conversely, suppose that every $(E,X,\pi)$ is a fibred-WCVP space, then this holds in particular for the trivial bundle $(\mathcal{E},X,p)$. Given that $f,g:X\longrightarrow X$ are commuting self-maps, we construct the pullback bundles $(f^*\mathcal{E},X,p_f)$ and $(g^*\mathcal{E},X,p_g)$, which fit into the following diagrams of commuting squares 
\[\begin{tikzcd} 
\mathcal{E} \arrow[r,"\varphi_1"] \arrow[d,"p"] & g^*\mathcal{E} \arrow[d,"p_g"] \arrow[r,"\varphi_g"] & \mathcal{E} \arrow[d,"p"] \\
X \arrow[r,"Id_X"] & X  \arrow[r,"g"] & X
\end{tikzcd}\qquad
\begin{tikzcd} 
\mathcal{E} \arrow[r,"\varphi_2"] \arrow[d,"p"] & f^*\mathcal{E} \arrow[d,"p_f"] \arrow[r,"\varphi_f"] & \mathcal{E} \arrow[d,"p"] \\
X \arrow[r,"Id_X"] & X  \arrow[r,"f"] & X
\end{tikzcd}\]
Since any pullback of the trivial bundle is isomorphic to it we have the existence of homeomorphisms $\varphi_1$ and $\varphi_2$ which restrict to fibre-wise isomorphisms. In particular, denoting $F_x=\{x\}\times F\subset\mathcal{E}$
\[\varphi_1(F_x)\subseteq\{x\}\times F_{g(x)}\qquad\mathrm{and}\qquad\varphi_2(F_x)\subseteq\{x\}\times F_{f(x)}\]
As mentioned in \Cref{defA7}, the pair $(\varphi_f,f)$ is a bundle morphism where $\varphi_f$ is the projection onto the second factor-- likewise for $(\varphi_g,g)$. Since $f$ and $g$ commute it is straightforward to verify that $\varphi_g\varphi_1,\varphi_f\varphi_2:\mathcal{E}\longrightarrow\mathcal{E}$ commute fibre-wise
\[\varphi_f\varphi_2\varphi_g\varphi_1(F_x)\asymp_pF_{fg(x)}=F_{gf(x)}\asymp_p\varphi_g\varphi_1\varphi_f\varphi_2(F_x)\] 
Moreover, $\varphi_g\varphi_1$ fixes a fibre $F_x$ if and only if $x$ is a fixed point of $g$, and analogously for $\varphi_f\varphi_2$.
It follows that there exists $F_z\in\mathcal{CV}(\varphi_g\varphi_1,\varphi_f\varphi_2)_{\mathrm{fibre}}$, and thus $\varphi_g\varphi_1(F_z)\asymp_pF_{g(z)}$ belongs to the same fibre as $F_{f(z)}\asymp_p\varphi_f\varphi_2(F_z)$, which immediately implies $z\in\mathcal{CV}(f,g)$.
\end{proof}
Note that $X=(X,X,Id_X)$ is a fibre bundle over itself with model fibre $F=\{pt\}$, hence the above provides a non-trivial generalization towards determining when $X$ does not possess the WCVP, particularly if $X$ admits non-trivial fibre bundles. To make \Cref{lemBa1} more useful with respect to positive characterization of WCVP spaces, we show it is only necessary to prove the fibred-WCVP for \textit{some} fibre bundle.

\begin{mycor}\label{corBa1}
A locally compact Hausdorff space $X$ is a WCVP space if and only if there exists a fibre bundle $(E,X,\pi)$ which has the \textit{fibred}-WCVP.
\end{mycor}
\begin{proof}
The only if direction follows obviously from \Cref{lemBa1}. For the other direction we only need to use local triviality and modify the argument at the end of \Cref{lemBa1} concerning the product bundle. 
\end{proof}
Unfortunately, the property of $(E,X,\pi)$ \textit{being} a fibred-WCVP space clearly depends on the entire fibre bundle structure, and thus determining whether $E$ has some coincidence value property as a topological space in its own right must naturally be a much more subtle question. Indeed by \Cref{corBa1} it is obvious that if $E=X_1\times X_2$ where $X_1$ is a WCVP space but $X_2$ is not, then $(E,X_1,p_1)$ is a fibred-WCVP bundle while $(E,X_2,p_2)$ fails to be so. In particular, this observation would suggest that $E$ being a non-WCVP space generally says nothing with regards to a given $(E,X,\pi)$ having the fibred-WCVP. Indeed we have the following elementary criterion, which is again negative in nature.

\begin{mypro}\label{proBa1}
A topological space $E$ fails to have the WCVP if there exists a fibre bundle $(E,X,\pi)$ over a locally compact and second countable Hausdorff space where $X$ does not possess the WCVP.
\end{mypro}
\begin{proof}
If $(E,X,\pi)$ does not have the fibred-WCVP then it is impossible that the stronger condition of $E$ being a WCVP space holds. However, by assumption, $X$ is not a WCVP space and thus the result follows immediately from \Cref{lemBa1}. 
\end{proof}
Note that the above results pertaining to the non-possession of the fibred-WCVP also apply to fibre bundles over the total space $E$ given that the topology of $E$ is nice enough. To this end, when we are also interested in the total space, $X$ is usually taken to be second countable in addition to locally compact and thus is paracompact. In fact, if $X$ is a non-WCVP space and there exists a sequence of fibre bundles
\[\begin{tikzcd} 
(E_1,X,\pi_1) \arrow[r] & (E_2,E_1,\pi_2) \arrow[r] & (E_3,E_2,\pi_3) \arrow[r] &  \cdots
\end{tikzcd}\]
the above proposition along with \Cref{corBa1} immediately gives that none of the total spaces $E_i$ possess the WCVP. The essential point of the above results is that regardless of the choice of fibre or bundle twisting, the total space inherits from the base space any obstructions to satisfying the WCVP. We are now going to turn our attention to obtaining some positive results concerning total spaces $E$ where $(E,X,\pi)$ is a fibre bundle over a WCVP space $X$. Since such a total space is constructed by gluing together multiple copies of $X$ in some appropriate way, we first prove the following lemma which is essentially the natural generalization of \Cref{exaA1}.

\begin{mylem}\label{lemBa2}
If $X_1$ and $X_2$ are connected Hausdorff spaces such that $X_1$ has the CVP and $X_2$ is a WCVP space, then the disjoint union $X_1\sqcup X_2$ possesses the WCVP. However, even if all $X_i$ are CVP spaces, $X=\bigsqcup_{i=1}^nX_i$ never possesses the WCVP for finite $n\geq3$.
\end{mylem}
\begin{proof}
By the connectedness assumption any self-map $f:X_1\sqcup X_2\longrightarrow X_1\sqcup X_2$ must have one of the following characterizations
\[\mathrm{(i)}\quad f:X_1\sqcup X_2\longrightarrow X_1\qquad\mathrm{(ii)}\quad f:X_i\longrightarrow X_i\]
\[\mathrm{(iii)}\quad f:X_1\sqcup X_2\longrightarrow X_2\qquad\mathrm{(iv)}\quad f:X_i\longrightarrow X_j\]
If $f,g:X_1\sqcup X_2\longrightarrow X_1\sqcup X_2$ are fixed-point-free commuting self-maps then they are either both of type (iii) or both of type (iv). In the first case we clearly have $f,g:X_2\longrightarrow X_2$ and thus $\mathcal{CV}(f,g)\neq\emptyset$. When $f$ and $g$ are of type (iv) we have commuting self-maps $f^2,fg:X_1\longrightarrow X_1$ which thus both possess fixed points; it follows that $x\in\mathcal{CV}(f^2,fg)$ which implies that $f(x)\in\mathcal{CV}(f,g)$. Now suppose that commuting self-maps $f,g:X_1\sqcup X_2\longrightarrow X_1\sqcup X_2$ have fixed points; then at least one map must be type (ii) and neither can be type (iv). Without loss of generality let $f$ be of type (ii); for any of the three possiblities for $g$ we have $f,g:X_1\longrightarrow X_1$ or $f,g:X_2\longrightarrow X_2$, hence $\mathcal{CV}(f,g)\neq\emptyset$.

Now suppose that $X=\bigsqcup_{i=1}^nX_i$ and $n\geq3$. The proof does not depend on the topological properties of the $X_i$, thus we may as well assume that $X=\{x_1,x_2,\ldots,x_n\}$ is a discrete set of $n$-points. Let $S_n$ denote the symmetric group on $n$-letters; since $n\geq3$ the $n$-cycle $\sigma=(12\cdots n)\in S_n$ possesses a square $\sigma^2$ which is an $n$-cycle distinct from $\sigma$. Clearly, $\sigma$ and $\sigma^2$ are commuting fixed-point-free self-maps of $X$, with $\sigma$ having order $n$; in particular, if $x_i\in\mathcal{CV}(\sigma,\sigma^2)$ then $\sigma^{n-1}(x_i)=\sigma^n(x_i)=x_i$, which is impossible. Translating this to the general case, we can always construct distinct fixed-point-free self-maps $f,f^2:X\longrightarrow X$ satisfying $f:X_i\longrightarrow X_{i+1(\mathrm{mod}\;n)}$.
\end{proof}

By \Cref{lemBa2}, the disjoint union of two copies of a connected CVP space $X$ has the WCVP, hence it is easy to see that the total space $X\times F$ of the trivial bundle is a WCVP space for any two-point discrete space $F=\{y_1,y_2\}$. More generally, since $\{pt\}\times X\simeq X$ 
\[(\{y_1\}\times X)\bigsqcup(\{y_2\}\times X)=X\times\{y_1,y_2\}=\bigsqcup_{x\in X}\{x\}\times\{y_1,y_2\}\] 
By the atlas/chart lemma, given a cover $\{U_\alpha\}$ of $X$, a unique fibre bundle structure $(E,X,\pi)$ with model fibre $F=\{y_1,y_2\}$ is determined by the group of transition functions 
\[\rho_{\alpha\beta}:U_\alpha\cap U_\beta\neq\emptyset\longrightarrow\mathsf{Homeo}(F)\cong\mathbb{Z}_2\]
Vary the product associated to each $x$ by replacing $\{x\}\times\{y_1,y_2\}$ with $F_x\simeq\{x\}\times\{y_1,y_2\}$, and define $E=\bigsqcup_{x\in X}F_x$ along with the obvious projection map $\pi:E\longrightarrow X$. The immediate question is to what extent coincidence value properties are preserved under these operations. Unsurprisingly, this is a much more subtle question to answer, though by careful technical use of the topological assumptions one is able to improve on the modest partial results given below.

\begin{mylem}\label{lemBa4}
Let $(E,X,\pi)$ be a fibre bundle with structure group $\mathbb{Z}_2$ and model fibre $F$ a discrete two-point space, where $X$ is a locally compact, second countable, and connected Hausdorff space. 
\begin{enumerate}[label={\bf(\roman{enumi})}] 
\item The total space $E$ is a uniformly-2-PCVP space with respect to bundle maps given that $X$ has the CVP.
\item If $E$ is a WCVP space, then so is $X$. Suppose that $X$ is a CVP space and every pair of fixed-point-possessing, commuting bundle maps on $E$ have a coincidence value. Then any two fixed-point-free pair of commuting bundle maps on $E$ possesses a coincidence value, given that their composition is not fixed-point-free. 
\end{enumerate}	
\end{mylem}
\begin{proof}
\noindent\textbf{(i)} Consider an arbitrary pair of commuting self-maps $\hat{f},\hat{g}:E\longrightarrow E$. We will assume that both of these have the structure of bundle self-morphisms 
\[(\hat{f},f),(\hat{g},g):(E,X,\pi)\longrightarrow(E,X,\pi)\]
where $\hat{f}:\pi^{-1}(\{x\})\longrightarrow\pi^{-1}(\{x'\})$ if and only if $f(x)=x'$. By the proof of \Cref{lemBa1} the fibred-CVP holds for $(E,X,\pi)$ if $X$ is a CVP space, thus there exists $z\in X$ such that $\pi^{-1}(\{z\})\in\mathcal{CV}(\hat{f},\hat{g})_{\mathrm{fibre}}$. Local triviality asserts that every fibre is a discrete two-point space, thus it is not difficult to see that $\hat{f},\hat{g}:\pi^{-1}(\{z\})\longrightarrow\pi^{-1}(\{w\})$ must possess a coincidence value unless they are of the form
\[\hat{f}(z_1)=\hat{g}(z_2)=w_i\quad\mathrm{and}\quad\hat{f}(z_2)=\hat{g}(z_1)=w_j\]
\[\hat{f}(z_1)=\hat{f}(z_2)=w_i\quad\mathrm{and}\quad\hat{g}(z_2)=\hat{g}(z_1)=w_j\]
In the first case we compose by $\hat{f}$ to obtain $\hat{f}^2(z_1)=\hat{f}\hat{g}(z_2)=\hat{g}\hat{f}(z_2)=\hat{g}^2(z_1)$, so $z_1,z_2\in\mathcal{CV}(\hat{f}^2,\hat{g}^2)$. In the second case we appeal to \Cref{proA3} to realize that since $z\in\mathcal{CV}(f,g)$ the point $w=f(z)$ is also a coincidence point, hence $\pi^{-1}(\{w\})\in\mathcal{CV}(\hat{f},\hat{g})_{\mathrm{fibre}}$ by \Cref{lemBa1}. It follows that $\hat{f},\hat{g}:\pi^{-1}(\{w\})\longrightarrow\pi^{-1}(\{y\})$, where commutativity implies $\hat{f}(w_j)=\hat{f}\hat{g}(z_1)=\hat{g}\hat{f}(z_1)=\hat{g}(w_i)$; either we have $\hat{f}(w_i)=\hat{g}(w_i)$-- which clearly provides the desired result-- or $\hat{f}(w_i)=\hat{g}(w_j)$, which reduces to the previously discussed first case.

\noindent\textbf{(ii)} Consider an arbitrary pair of fixed-point-free commuting self-maps $\hat{f},\hat{g}:E\longrightarrow E$ which are also endowed with a bundle morphism structure 
\[(\hat{f},f),(\hat{g},g):(E,X,\pi)\longrightarrow(E,X,\pi)\] 
Since $X$ is an FPP space, again appealing to \Cref{lemBa1}, there exists $y,x\in X$ such that $\pi^{-1}(\{x\})\in\mathcal{FP}(\hat{f})_{\mathrm{fibre}}$ and $\pi^{-1}(\{y\})\in\mathcal{FP}(\hat{g})_{\mathrm{fibre}}$. By definition, this means that $\hat{f}:\pi^{-1}(\{x\})\longrightarrow\pi^{-1}(\{x\})$, hence it is necessary that $\hat{f}(x_1)=x_2$ and $\hat{f}(x_2)=x_1$, otherwise it is clear that $\hat{f}$ must fix one of these points. It follows that $x_1,x_2\in\mathcal{FP}(\hat{f}^2)$, and analogously we have $y_1,y_2\in\mathcal{FP}(\hat{g}^2)$. By hypothesis, $\hat{f}\hat{g}$ possesses a fixed point, thus there exists $z\in\mathcal{CV}(\hat{f}\hat{g},\hat{f}^2)$ and so $\hat{f}(z)\in\mathcal{CV}(\hat{g},\hat{f})$.

Conversely, $E$ being a WCVP space with respect to bundle maps is by definition equivalent to saying it is a fibred-WCVP space, hence the result follows from \Cref{corBa1}.
\end{proof}

The results of the this section, particularly the discussion preceding \Cref{proBa1}, raise the natural question as to how much the fibre structure influences the coincidence value properties of the total space. Indeed, we would wish to obtain a general answer as to in which instances the positive result concerning principal $\mathbb{Z}_2$-bundles can be extended.
\begin{myrem}
By \Cref{lemBa2} it is immediate that the total space $X\times F$ of any trivial bundle, with model fibre $|F|\geq3$ being a discrete finite space, cannot be a WCVP space. Since every fibre bundle is locally trivial we would surmise that the general situation is not promising, though we refrain from giving a complete treatment in this paper. On the other hand, assuming the discreteness of $F$, it is not difficult to prove such a statement for a principal $\Gamma$-bundle given that $F=\Gamma$ is of at least countable cardinality. Indeed, if we are willing to make use of heavier machinery, analogous partial results can be stated for $F=\Gamma$ a general topological group and $(E,X,\pi)$ a principal $\Gamma$-bundle. This shall be investigated in the following section.	
\end{myrem}

By definition, a fibre bundle results from the global twisting of the disjoint union of copies of the base space; another ubiquitous way of obtaining ``topologically twisted" spaces from disjoint unions is the formation of adjunction spaces. The following result is an important generalization of \Cref{exaB1}, inspired by the fact that $\mathbf{S}^1$ is an adjunction space formed from $[0,1]$ and $\mathbf{S}^2$ is an adjunction space formed from $\mathbf{D}$. For $X_1$ and $X_2$ both Hausdorff and connected, we shall quickly recall some important properties of the adjunction space $X_1\sqcup_{h,A} X_2$ and the quotient map $q:X_1\sqcup X_2\longrightarrow X_1\sqcup_{h,A} X_2$. Taking $A\subset X_1$ to be closed implies the following elementary facts (see \cite[Chapter 9]{SW04}) 
\begin{enumerate}
\item $X_1\sqcup_{h,A} X_2$ is Hausdorff and connected.
\item $q|_{(X_1\setminus A)}$ is an open embedding while $q|_{X_2}$ is a closed embedding.
\end{enumerate}
Moreover, in the case that $h$ is a homeomorphism onto its image we can write
\[X_1\sqcup_{h,A} X_2=X_2\sqcup_{h^{-1},h(A)} X_1\]
and so there exists a symmetry which allows us to identify $q(A)=q(h(A))$ and to freely swap $X_1$ with $X_2$ in the above results: namely, $q|_{(X_2\setminus h(A))}$ and $q|_{X_1}$ are open and closed embeddings respectively. Importantly, this tells us that any homeomorphism $\phi:X_1\longrightarrow X_2$ induces a homeomorphism $\hat{\phi}:q(X_1)\longrightarrow q(X_2)$. Furthermore, imposing a non-cut property of $A$ and $h(A)$ implies connectedness of $X_1\setminus A$ and $X_2\setminus h(A)$, hence also of the image of these sets under $q$. This provides the important fact that for any connected subset $U\subseteq X_1\sqcup_{h,A} X_2$ 
\begin{equation}
U\cap q(A)=\emptyset\implies\quad U\subseteq q(X_1\setminus A)\quad\mathrm{or}\quad U\subseteq q(X_2\setminus h(A))
\end{equation}
Supposing otherwise would imply existence of a decomposition of $U$ into non-empty open sets $U_1=U\cap q(X_1\setminus A)$ and $U_2=U\cap q(X_2\setminus h(A))$, which are clearly disjoint. 

\begin{mylem}\label{lemBa3}
Let $X_1$ and $X_2$ be homeomorphic connected Hausdorff spaces, where $A\subseteq X_1$ is compact and $h:A\longrightarrow X_2$ is a homeomorphism onto its image.
\begin{enumerate}[label={\bf(\roman{enumi})}] 
\item If every pair of non-surjective self maps of the adjunction space $X_1\sqcup_{h,A} X_2$  possesses a coincidence point, then $X_1$ (hence $X_2$) is a CVP space. More generally, if $X_1\sqcup_{h,A} X_2$ possesses the WCVP (PCVP) with respect to non-surjective maps, then $X_1$ is a WCVP (PCVP) space.
\item Assuming $X_1$ is a CVP space and every proper, closed, and connected subset $K\subset X_1\sqcup_{h,A} X_2$ is contained within a region homoeomorphic to $X_1$, then an arbitrary self-map $f$ of $X_1\sqcup_{h,A} X_2$ shares a coincidence point with every non-surjective $g$ which commutes with it. Under the same subspace stipulations, given that $X_1$ is a WCVP (PCVP) space, $X_1\sqcup_{h,A} X_2$ possesses the WCVP (PCVP) with respect to one map being non-surjective.
\end{enumerate}
\end{mylem}
\begin{proof}
\noindent\textbf{(i)} The forward implication requires almost none of the connectivity and compactness assumptions. Given commuting self-maps $f,g:X_1\longrightarrow X_1$ and a homeomorphism $\phi:X_1\longrightarrow X_2$, we can construct the self-maps $\hat{f},\hat{g}:X_1\sqcup_{h,A} X_2\longrightarrow X_1\sqcup_{h,A} X_2$   
\[\hat{f}=\left\{\begin{array}{cc}
q|_{X_1}fq|_{X_1}^{-1} & \mathrm{for}\quad q(x)\in q(X_1) \\
q|_{X_1}f\phi^{-1}q|_{X_2}^{-1} & \mathrm{for}\quad q(y)\in q(X_2)
\end{array}\right.\qquad
\hat{g}=\left\{\begin{array}{cc}
q|_{X_1}gq|_{X_1}^{-1} & \mathrm{for}\quad q(x)\in q(X_1) \\
q|_{X_1}g\phi^{-1}q|_{X_2}^{-1} & \mathrm{for}\quad q(y)\in q(X_2)
\end{array}\right.\]
These maps are well defined and continuous since $q(A)=q(h(A))$, $q|_{X_1}$ is a homeomorphism, and by hypothesis we have $h\equiv\phi_{|A}$; moreover, since $\hat{f},\hat{g}:X_1\sqcup_{h,A} X_2\longrightarrow q(X_1)$ it is easy to see that the maps commute and thus $\mathcal{CV}(\hat{f},\hat{g})\neq\emptyset$.
\[\hat{f}\hat{g}(q(x))=q|_{X_1}fq|_{X_1}^{-1}\circ q|_{X_1}gq|_{X_1}^{-1}(q(x))=q|_{X_1}fgq|_{X_1}^{-1}=q|_{X_1}gfq|_{X_1}^{-1}=\hat{g}\hat{f}(q(x))\]
\[\hat{f}\hat{g}(q(y))=q|_{X_1}fq|_{X_1}^{-1}\circ q|_{X_1}g\phi^{-1}q|_{X_1}^{-1}(q(y))=q|_{X_1}fg\phi^{-1}q|_{X_1}^{-1}=q|_{X_1}gf\phi^{-1}q|_{X_1}^{-1}=\hat{g}\hat{f}(q(y))\]
It is similarly clear that $f$ and $g$ are fixed-point-free if and only if $\hat{f},\hat{g}$ are fixed-point-free: $x\in\mathcal{FP}(f)\Leftrightarrow q(x)\in\mathcal{FP}(\hat{f})$. If $q(x)\in\mathcal{CV}(\hat{f},\hat{g})$, then $q|_{X_1}fq|_{X_1}^{-1}(q(x))=q|_{X_1}gq|_{X_1}^{-1}(q(x))$ and $q|_{X_1}$ being a homeomorphism implies that $f(x)=g(x)$; if $q(y)\in\mathcal{CV}(\hat{f},\hat{g})$, then $q|_{X_1}f\phi^{-1}q|_{X_1}^{-1}(q(y))=q|_{X_1}g\phi^{-1}q|_{X_1}^{-1}(q(y))$ and so $f(\phi^{-1}(y))=g(\phi^{-1}(y))$. This proves the desired results concerning $X_1$ being a CVP or WCVP space. Now for the PCVP case it follows from the definitions that 
\[\hat{f}^n=\left\{\begin{array}{cc}
q|_{X_1}f^nq|_{X_1}^{-1} & \mathrm{for}\quad q(x)\in q(X_1) \\
q|_{X_1}f^n\phi^{-1}q|_{X_2}^{-1} & \mathrm{for}\quad q(y)\in q(X_2)
\end{array}\right.\]
hence analogous argumentation as before proves that $q(x)\in\mathcal{CV}(\hat{f}^n,\hat{g}^m)$ implies that $x\in\mathcal{CV}(f^n,g^m)$ and $q(y)\in\mathcal{CV}(\hat{f}^n,\hat{g}^m)$ implies that $\phi^{-1}(y)\in\mathcal{CV}(f^n,g^m)$.

\noindent\textbf{(ii)} Suppose first that $f,g:X_1\sqcup_{h,A} X_2\longrightarrow X_1\sqcup_{h,A} X_2$ are commuting self-maps which are either both fixed-point-free or possess a fixed point. Under the assumption that $f$ is not surjective, the connected and compact set
\[K=f(X_1\sqcup_{h,A} X_2)\] 
is a proper subset of the adjunction space. Using the fact that $q_{|X_1}$ is an embedding, by hypothesis there exists a homeomorphism $\phi$ such that $X_1\simeq\phi(q(X_1))$ and $K\subset\phi(q(X_1))$. By restricting to $\phi(q(X_1))$ we thus obtain a pair of commuting self-maps
\[f^2,fg:\phi(q(X_1))\longrightarrow\phi(q(X_1))\]
where \Cref{thmA1} implies that $\phi(q(X_1))$ possesses the WCVP, and the usual argument shows that $f(\phi(q(x)))\in\mathcal{CV}(f,g)$ for some $\phi(q(x))\in\mathcal{CV}(f^2,fg)$. The exact same reasoning goes through if we drop the fixed point restrictions of $f$ and $g$, and instead assume that $X_1$ is a CVP space.

Given that $X_1$ is a PCVP space, then \Cref{thmA1} implies that $\phi(q(X_1))$ possesses the PCVP. Again restricting to $\phi(q(X_1))$, for any commuting pair $f,g:X_1\sqcup_{h,A} X_2\longrightarrow X_1\sqcup_{h,A} X_2$ such that $f$ is non-surjective, we obtain that for any $k\geq2$
\[f^k,fg:\phi(q(X_1))\longrightarrow\phi(q(X_1))\]
It follows that $\mathcal{CV}(f^{kn},(fg)^m)\neq\emptyset$, and so we have $f^{kn}(q(x))=(fg)^m(q(x))$ for some $q(x)$; choosing $k$ such that $kn>m$ this is equivalent to $f^{kn-m}(f^m(q(x)))=g^m(f^m(q(x)))$, thus $\mathcal{CV}(f^{kn-m},g^m)\neq\emptyset$.
\end{proof}
We remark that the proof of the above theorem clearly shows that the notion of PCVP can be replaced with its uniform counterpart in every instance. We are now in a position to use the combined power of the results proved in this section to investigate coincidence value properties of a few important spaces.

\begin{mypro}\label{proBa2}
$\mathbf{S}^3$ is neither a PCVP nor WCVP space.
\end{mypro}
\begin{proof}
The key observation is that there is a Heegaard splitting of $\mathbf{S}^3$ which provides the adjunction space representation $\mathbf{S}^3\simeq\mathbf{S}^1\times\mathbf{D}\sqcup_{h,\mathbf{S}^1\times\mathbf{S}^1}\mathbf{D}\times\mathbf{S}^1$ where $h$ is the ``identity map" sending each point on the boundary of the first solid torus to the corresponding boundary point on the second solid torus. If $\mathbf{S}^3$ were a PCVP or WCVP space, then by part (i) of \Cref{lemBa3} the same would be true for $\mathbf{S}^1\times\mathbf{D}$. However, $\mathbf{S}^1\times\mathbf{D}$ cannot possess the WCVP; otherwise, viewing it as the total space of a fibre bundle over $\mathbf{S}^1$, by \Cref{corBa1} this property would be inherited by the circle. Similarly, $\mathbf{S}^1\times\mathbf{D}$ cannot possess the PCVP due to the existence of commuting self-maps $f=(\theta_1,Id_{\mathbf{D}})$ and $g=(\theta_2,Id_{\mathbf{D}})$, where $\theta_i$ is an irrational rotation.
\end{proof}

\begin{mythm}\label{thmBa2}
The closed unit disk $\mathbf{D}$ has the PCVP if and only if there exists $m,n\in\mathbb{N}\setminus\{0\}$such that $\mathcal{CV}(f^n,g^m)\neq\emptyset$ for every commuting pair $f,g:\mathbf{S}^2\longrightarrow\mathbf{S}^2$ where $f$ is not surjective. If the real projective plane $\mathbb{R}\mathbf{P}^2$ is a CVP space, then $\mathbf{S}^2$ is a uniformly-2-PCVP space with respect to morphisms.
\end{mythm}
\begin{proof}
There is a homeomorphism $\mathbf{S}^2\simeq\mathbf{D}\sqcup_{h,\mathbf{S}^1}\mathbf{D}$ where $h$ is the ``identity map" sending each point on the boundary of the first disk to the corresponding boundary point on the second disk. The forward implication thus follows immediately from part (i) of \Cref{lemBa3}. To prove the reverse implication, note that for every closed and connected $K\subsetneq\mathbf{S}^2$ along with a choice of $x\in K^c$, there exist disjoint open balls $B_K$ and $B_x$ containing $K$ and $x$ respectively. By regularity of $\mathbf{S}^2$ we can find a closed ball $\overline{B'}_x\subset B_x$, hence $x\notin\overline{B}_K\simeq\mathbf{D}$, which satisfies the hypothesis of part (ii) of \Cref{lemBa3}.

Since $(\mathbf{S}^2,\mathbb{R}\mathbf{P}^2,\pi)$ is a principal $\mathbb{Z}_2$-bundle, then the desired coincidence value relationship follows immediately from part (i) of \Cref{lemBa4}.
\end{proof}

\begin{mythm}\label{thmBa3}
The closed unit disk has the CVP if and only if $\mathcal{CV}(f,g)\neq\emptyset$ for every commuting pair $f,g:\mathbf{S}^2\longrightarrow\mathbf{S}^2$ such that $f$ is not surjective. The real projective plane is a CVP space if $\mathbf{S}^2$ is a WCVP space with respect to bundle maps; conversely, if $\mathbb{R}\mathbf{P}^2$ is a CVP space and every pair of fixed-point-possessing, commuting bundle self-maps on $\mathbf{S}^2$ possess a coincidence value, then $\mathbf{S}^2$ is a WCVP space with respect to bundle maps. 
\end{mythm}
\begin{proof}
Using the same argumentation as in \Cref{thmBa2} and noting that WCVP is equivalent to CVP for the disk, the first result is a consequence of \Cref{lemBa3}. Recall that if $f,g:\mathbf{S}^2\longrightarrow\mathbf{S}^2$ are both fixed-point-free, then $\mathcal{FP}(fg)\neq\emptyset$ since $\mathsf{deg}(f)=\mathsf{deg}(f)=-1$ and so $\mathsf{deg}(fg)=1\neq-1$. Hence, since $(\mathbf{S}^2,\mathbb{R}\mathbf{P}^2,\pi)$ is a principal $\mathbb{Z}_2$-bundle, the desired coincidence value relationships between the sphere and projective plane follow immediately from part (ii) of \Cref{lemBa4}.
\end{proof}

\subsection{$\Gamma$-equivariant maps}\label{secB2}

Using the extra structure that comes from a topological space admitting a continuous $\Gamma$-action, we can prove sufficient conditions for a locally compact Hausdorff space to not possess the WCVP.
\begin{mythm}\label{thmBb1}
Let $f,g:X\longrightarrow X$ be distinct commuting self-maps on a locally compact Hausdorff space which admits a $\Gamma$-action; $X$ cannot have the WCVP if either of the following hold. 
\begin{enumerate}[label={\bf(\roman{enumi})}] 
\item $X$ admits a transitive $\Gamma$-action with $f$ and $g$ being $\Gamma$-equivariant.
\item $X$ admits a proper $\Gamma$-action for some non-compact $\Gamma$, with $f$ and $g$ being $\Gamma$-equivariant, and $\mathcal{CV}(f,g)$ being contained in a compact subset.
\end{enumerate}
\end{mythm}
\begin{proof}
\textbf{(i)} Firstly note that if $f$ is $\Gamma$-equivariant but not the identity function then $f$ must be fixed-point-free. Indeed, if there exists a point such that $f(x)=x$, then by the transitive nature of the $\Gamma$-action for each $y\in X$ there exists $\gamma$ such that
\[f(y)=f(\gamma\cdot x)=\gamma\cdot f(x)=\gamma\cdot x=y\]
and thus $f\equiv Id_X$. An analogous argument shows that if $f$ and $g$ are both $\Gamma$-equivariant and $\mathcal{CV}(f,g)\neq\emptyset$ then $f\equiv g$. Thus any distinct commuting $f,g:X\longrightarrow X$ provide an example of a fixed-point-free pair which cannot have a coincidence value.

\textbf{(ii)} Let us first assume that $\mathcal{CV}(f,g)$ is compact. That $\mathcal{CV}(f,g)$ is a proper subset of $X$ is obvious, otherwise $f\equiv g$. By equivariance of $f$ and $g$, for every $\alpha\in \Gamma$ the map $f_\alpha$ commutes with both $f$ and $g$, so $f_\alpha:\mathcal{CV}(f,g)\longrightarrow\mathcal{CV}(f,g)$ by \Cref{proA3}. Properness of the $\Gamma$-action implies that the following set is compact.
\[\Gamma_{\mathcal{CV}(f,g)}=\{\alpha\in \Gamma:\;\alpha\cdot\mathcal{CV}(f,g)\cap\mathcal{CV}(f,g)\neq\emptyset\}\]
However, since $\Gamma_{\mathcal{CV}(f,g)}=\Gamma$ this clearly contradicts the fact that $f_\alpha(x)=\alpha\cdot x\in\mathcal{CV}(f,g)$ for all elements of the non-compact topological group $\Gamma$. It follows that $\mathcal{CV}(f,g)$ must be empty and so $f$ and $g$ are commuting maps with no coincidence value. This proof immediately generalizes to the case where $\mathcal{CV}(f,g)$ is contained in a compact set since closed subsets are also compact. Moreover, there cannot exist any compact $K$ such that $K\cap\mathcal{CV}(f,g)\neq\emptyset$ and $f(K)\subseteq K$; otherwise, $K\cap\mathcal{CV}(f,g)=\mathcal{CV}(f_{|K},g_{|K})$ is compact and the same argument applies. In particular, since every singleton is compact, $f$ (and hence $g$) must be fixed-point-free on $\mathcal{CV}(f,g)$. Furthermore, taking $f$ to be $Id_X$ implies that apart from $Id_X$ there can exist no other $H$-equivariant map with a fixed point; otherwise, for some $g$ distinct from the identity $\mathcal{CV}(Id_X,g)=\mathcal{FP}(g)\neq\emptyset$ contradicts the proper action of $\Gamma$.
\end{proof}

\Cref{thmBb1} provides us with more than might initially be thought, due to the fact that several important classes of topological spaces are locally compact Hausdorff and admit a transitive $\Gamma$-action in one form or another. In particular, every locally compact topological group $G$ admits a regular action on itself by left multiplication. If $X$ is a topological manifold then it admits a transitive action through $\mathsf{Homeo}(X)$, its group of homeomorphisms. Moreover, if $\widetilde{X}$ is a universal cover of some  manifold $X$ with non-trivial fundamental group $\pi_1(X)$ then $\widetilde{X}$ also admits a proper action by the group of deck transformations of $X$. Since the universal cover is just the special case of a principal $\pi_1(X)$-bundle we can state the following useful result.
\begin{mycor}\label{corBb1}
If $G$ is a topological group with center $Z(G)$ of cardinality greater than two, then $G$ is not a WCVP space. If $X$ is any suitable space and $(E,X,\pi)$ a principal $\Gamma$-bundle such that $Z(\Gamma)$ is of cardinality greater than two and $\Gamma$ is non-compact, then $E$ is not a WCVP space.
\end{mycor}
\begin{proof}
$G$ admits a regular $G$-action where the cardinality assumptions guarantee existence of distinct, non-trivial $\gamma_1,\gamma_2\in Z(G)$. The continuous functions $f_{\gamma_1},f_{\gamma_2}:G\longrightarrow G$ defined by $f_{\gamma_i}(x)=\gamma_i\cdot x$ are thus $G$-equivariant; the result follows from part (i) of \Cref{thmBb1}. Now, by definition as the total space of a principal $\Gamma$-bundle, $E$ admits a proper $\Gamma$-action. Since $\Gamma$ is non-compact, the cardinality assumptions again guarantee existence of distinct, non-trivial $\gamma_1,\gamma_2\in Z(\Gamma)$; the continuous functions $f_{\gamma_1},f_{\gamma_2}:E\longrightarrow E$ are thus $\Gamma$-equivariant and the result follows from part (ii) of \Cref{thmBb1}.
\end{proof}

In particular, since $\mathbf{S}^1$ and $\mathbb{R}$ are Abelian Lie groups, this gives us another proof-- without considering explicit rotations or translations-- that the circle and real line are not WCVP spaces. Alternatively we can argue that since $\mathbb{R}$ is the universal cover of $\mathbf{S}^1$ and $\pi_1(\mathbf{S}^1)\cong\mathbb{Z}$, then the real line does not possess the WCVP. 
\begin{myex}\label{exaBb1}
The $n$-torus $(\mathbf{S}^1)^{\times n}$ is never an WCVP space for $n\geq1$
\end{myex} 
\begin{proof}
Since the $n$-torus is an uncountable Abelian Lie group it clearly satisfies the conditions of \Cref{corBc1};  it can also be viewed as a trivial principal $\mathbf{S}^1$-bundle over $\mathbf{S}^1$ thus not a WCVP space by \Cref{proBa1}.
\end{proof}

\begin{myrem}
It is well known that $\mathbf{S}^1,\mathbf{S}^3$ and are the only spheres which are topological groups; moreover, the proof for part (i) of \Cref{thmBb1} does not require associativity and so generalizes to the $H$-space $\mathbf{S}^7$ acting on itself. Since $\mathbf{S}^1$ is not a WCVP space then by \Cref{proBa1} neither is the total space $\mathbf{S}^2\times\mathbf{S}^1$ when viewed as a topological space in its own right. Considering the collection of Hopf fibrations 
\[\mathbf{S}^1\hookrightarrow\mathbf{S}^3\longrightarrow\mathbf{S}^2\qquad\mathbf{S}^3\hookrightarrow\mathbf{S}^7\longrightarrow\mathbf{S}^4\qquad\mathbf{S}^7\hookrightarrow\mathbf{S}^{15}\longrightarrow\mathbf{S}^8 \]	
it is immediate that if the non-WCVP lifts from the trivial bundle, then this would imply that $\mathbf{S}^3$ is not a WCVP space either, and repeating the argument with respect to $\mathbf{S}^4\times\mathbf{S}^3$ we would expect the same result for $\mathbf{S}^7$ and subsequently for $\mathbf{S}^{15}$. By part (i) of \Cref{thmBb1} this could be accomplished by constructing $\mathbf{S}^n$-equivariant self-maps $f,g:\mathbf{S}^n\longrightarrow\mathbf{S}^n$ for $n=3,7$, though it is not possible to reduce the argument to an application of \Cref{corBb1}. Indeed since $\mathbf{S}^3$ is isomorphic as a group to $SU(2)$ so we thus have $|Z(\mathbf{S}^3)|=|Z(SU(2))|=2$; it is also not possible to directly apply part (ii) of \Cref{thmBb1} due to the compactness of the involved spaces.	
\end{myrem}

\section{Sheaf Cohomology for PCVP spaces}\label{secB'}
Our first observation is that for any two distinct (commuting) self maps $f,g:X\longrightarrow X$ there must exist a point $x_0\in X$ satisfying $f(x_0)\neq g(x_0)$. For a Hausdorff space $X$ the subset $A=\{x\in X|\;f(x)=g(x)\}$ is closed, hence $A^c\ni x_0$ is an open set and there exists an open neighbourhood $N_{x_0}$ such that $f(y)\neq g(y)$ for all $y\in N_{x_0}$. Loosely speaking, this shows that every Hausdorff space is ``locally" non-CVP, failing to possess the CVP ``globally" if there exists a pair $f,g$ along with an open cover of neighbourhoods $\{N_{x_j}^{f,g}\}_{j\in J}$. It is this local to global gluing which we seek to make rigorous in the remainder of this section. 

\subsection{Constructing A Suitable Presheaf}\label{secB3}
\begin{mypro}\label{proBc1}
Consider two arbitrary-- but distinct-- commuting self-maps $f,g:X\longrightarrow X$ of a Hausdorff space, and let $h_i,h_j\in\{f,g,Id_X\}$. The set $S^x_{f,g}$ of all pairs generated by $(h_i,h_j)$ forms an Abelian monoid with respect to the operation $(h_1,h_2)+(h_3,h_4)=(h_1h_3,h_2h_4)$
modulo the equivalence relations
\begin{equation*}
\begin{gathered}
(f^m,Id_X)\sim(Id_X,f^m)\sim(g^n,Id_X)\sim(Id_X,g^n)\sim(Id_X,Id_X)\quad\mathrm{if}\;x\in\mathcal{CV}(f^m,g^n)\\
(f,Id_X)\sim(Id_X,f)\sim(g,Id_X)\sim(Id_X,g)\sim(Id_X,Id_X)\quad\mathrm{if}\;f\in\mathsf{Homeo}(X)
\end{gathered}
\end{equation*}

In particular, $S^x_{f,g}\cong\mathbb{N}^4$ or $S^x_{f,g}$ is trivial if $x\notin\mathcal{CV}(f^m,g^n)$ for all $m,n\in\mathbb{N}\setminus\{0\}$ while $S^x_{f,g}\cong\mathbb{Z}^2_m\oplus\mathbb{Z}^2_n$ if there exists a minimal pair $m,n\in\mathbb{N}\setminus\{0\}$ such that $x\in\mathcal{CV}(f^m,g^n)$.
\end{mypro}
\begin{proof}
The addition operation is clearly commutative since the maps $f,g$ and $Id_X$ all commute; associativity follows from the associativity of composition of maps and the identity element of $S^x_{f,g}$ is easily verified to be $0_S=(Id_X,Id_X)$. The definition of addition within $S^x_{f,g}$ implies that $m(h_i,h_j)=(h_i^m,h_j^m)$ for any natural number $m$, hence any arbitrary element $(f^{k_1}g^{k_2},f^{k_3}g^{k_4})\in S^x_{f,g}$ has the following decomposition
\begin{equation}\label{lindecomp}
k_1(f,Id_X)+k_2(g,Id_X)+k_3(Id_X,f)+k_4(Id_X,g)\qquad k_i\in\mathbb{N}
\end{equation}
By $\mathrm{ord}(f)\in\mathbb{N}\setminus\{0\}$ we refer to the minimal value such that $f^{\mathrm{ord}(f)}\equiv Id_X$, and we note that if either $\mathrm{ord}(f)$ or $\mathrm{ord}(g)$ is finite then the associated map is a homeomorphism with continuous inverse $f^{\mathrm{ord}(f)-1}$ (or $g^{\mathrm{ord}(g)-1}$). It follows from the above decomposition and second equivalence relation that $S^x_{f,g}\cong\{0_S\}$ unless $\mathrm{ord}(g)$ and $\mathrm{ord}(f)$ are infinite. Without loss of generality, we will henceforth assume that both $\mathrm{ord}(g)$ and $\mathrm{ord}(f)$ are infinite. Considering the case where $x\notin\mathcal{CV}(f^m,g^n)$ for all $m,n\in\mathbb{N}\setminus\{0\}$, the decomposition of \eqref{lindecomp} shows that the four summands are distinct and form a minimal cardinality generating set for $S^x_{f,g}$; moreover, this implies that $S^x_{f,g}$ is torsion-free if and only if these generators are linearly independent. Supposing that linear independence does not hold, there exist non-zero $k_1,k_2$ such that
\[k_1(f,Id_X)+k_2(g,Id_X)=(f^{k_1}g^{k_2},Id_X)=(Id_X,Id_X)\]
However, this implies that $f\circ f^{k_1-1}g^{k_2}=f^{k_1-1}g^{k_2}\circ f\equiv Id_X$; thus $f$ is a homeomorphism and $S^x_{f,g}\cong\{0_S\}$. If the generators are linearly independent they define a commutative, torsion-free monoid which is easily proven to be isomorphic to $\mathbb{N}^4$.

Now suppose there exists a minimal pair $m,n\in\mathbb{N}\setminus\{0\}$ \footnote{In the sense that we pick the smallest $(m,n)$ pair when ordered in lexicographical fashion} such that $x\in\mathcal{CV}(f^m,g^n)$. Since $(Id_X,f^m)\sim(Id_X,Id_X)$ the inverse to $(Id_X,f)$ is given by $(Id_X,f^{m-1})$, and likewise  $(Id_X,g)$ has inverse $(Id_X,g^{n-1})$. The decomposition provided in \eqref{lindecomp} thus shows that every element in $S^x_{f,g}$ possesses an inverse, and so $S^x_{f,g}$ is a finite Abelian group.
\end{proof}

\begin{mydef}\label{defBc1}
For every open subset $U$ of a Hausdorff space $X$ we can construct the Abelian monoid $S^U_{f,g}$ analogously to $S^x_{f,g}$, where the first equivalence relation is according to the stipulation $\mathcal{CV}(f^m,g^n)\cap U\neq\emptyset$. The group $\mathcal{S}^U_{f,g}$ is defined to be the Grothendieck completion of $S^U_{f,g}$, and we define
\[\mathcal{S}_U=\bigoplus_{(f,g)\in \widetilde{C}(X,X)}\mathcal{S}^U_{f,g}\]
where $\widetilde{C}(X,X)$ denotes the set of all pairs of continuous and commuting self-maps of $X$.
\end{mydef}
Note that in the special case of $U=\emptyset$ it is trivially true that $U\subseteq\mathcal{CV}(f,g)$ for every pair of self-maps, hence $\mathcal{S}^\emptyset_{f,g}\cong\{0\}$. If $U$ is non-empty and $f,g\notin\mathsf{Homeo}(X)$ we obtain analogous results to that of \Cref{proBc1}, namely 
\begin{equation}
\mathcal{S}^U_{f,g}\cong\left\{\begin{array}{cc}
\mathbb{Z}^4 & \mathrm{if}\quad\mathcal{CV}(f^m,g^n)\cap U=\emptyset,\;\forall m,n\in\mathbb{N}\setminus\{0\}\\
\mathbb{Z}^2_m\oplus\mathbb{Z}^2_n & \mathrm{if}\quad\mathcal{CV}(f^m,g^n)\cap U\neq\emptyset\quad\mathrm{for\;some\;minimal\;pair}\;(m,n)
\end{array}\right.
\end{equation}

Unfortunately, the above definition requires a few adjustments in order to be an even mildly useful computational tool. Recall that the $\mathsf{Homeo}(X)$ equivalence relation in \Cref{proBc1} implies triviality of $\mathcal{S}^U_{f,g}$ if either $f$ or $g$ is a homeomorphism; this forces us to impose conditions on $X$ which are essential in ensuring that this non-contribution does not lead to any loss of global information. For example, if $X$ is an FPP space we know that given any commuting pair $f,g:X\longrightarrow X$ where $f$ is a homeomorphism, then $x\in\mathcal{FP}(f^{-1}g)$ is equivalent to $x\in\mathcal{CV}(f,g)$ and indeed we should expect $\mathcal{S}^U_{f,g}$ to be trivial. This is of course, a very restrictive condition to impose, removing several interesting spaces-- including the even dimensional spheres-- from consideration. Indeed, taking direct inspiration from \Cref{proB1} we will impose only the requirement that given any commuting pair $f,g:X\longrightarrow X$ where $f\in\mathsf{Homeo}(X)$ then there exists $m,n\neq0$ such that $\mathcal{FP}(f^{-m}g^n)\neq\emptyset$. It is immediate that this implies $\mathcal{CV}(f^m,g^n)\neq\emptyset$, thus we need not worry about the loss of $\mathsf{Homeo}(X)$ information in passing to the monoid construction. The trade-off is that the equivalence relation renders $\mathcal{S}^U_{f,g}$ trivial rather than being the expected $\mathbb{Z}^2_m\oplus\mathbb{Z}^2_n$ if the contribution of homeomorphisms were recognized, though we shall show that this does not cause any significant issues with respect to determining PCVP. A more serious issue is that $\mathcal{S}_U$ has almost no hope of being countably generated, much less finitely generated as a $\mathbb{Z}$-module. Also problematic is the fact that the choice of minimal pair $(m,n)$ can vary widely between open subsets of $U$, thus there does not exist any easy way to determine existence of non-trivial morphisms between $\mathcal{S}^U_{f,g}$ and $\mathcal{S}^V_{f,g}$, even in the case $V\subset U$. 

As a first step towards rectifying the latter shortcomings, note that $\mathcal{S}^U_{f,g}=\mathcal{S}^U_{g,f}$ for each $(f,g)\in \widetilde{C}(X,X)$ and $\mathcal{S}^U_{f,g}\cong\{0\}$ if either $f$ or $g$ belong to $\mathsf{Homeo}(X)$. Hence for every open subset $U$ we have the following partitioning of $\widetilde{C}(X,X)$:
\begin{equation}
\widetilde{C}(X,X)_{\mathsf{Homeo}(X)}\sqcup\widetilde{C}(X,X)^U_{(\infty,\infty)}\sqcup\bigsqcup_{\substack{(m,n)\in\mathbb{N}\times\mathbb{N}\\0<n\leq m}}\widetilde{C}(X,X)^U_{(m,n)}	
\end{equation}
Here $\widetilde{C}(X,X)_{\mathsf{Homeo}(X)}$ denotes the set of all pairs of commuting self-maps of $X$ where at least one map belongs to $\mathsf{Homeo}(X)$. Given that $f,g\notin\mathsf{Homeo}(X)$, $\widetilde{C}(X,X)^U_{(\infty,\infty)}$ denotes those pairs for which $\mathcal{S}^U_{f,g}\cong\mathbb{Z}^4$. By $\widetilde{C}(X,X)^U_{(m,n)}$ we denote the set of all pairs such that $f,g\notin\mathsf{Homeo}(X)$ and $\mathcal{S}^U_{f,g}\cong\mathbb{Z}^2_m\oplus\mathbb{Z}^2_n$. This separates $\widetilde{C}(X,X)$ into equivalence classes; moreover, it follows that for each fixed $(m,n)$ and for all $(f,g)\in\widetilde{C}(X,X)^U_{(m,n)}$ the groups $\mathcal{S}^U_{f,g}$ are all equal. Giving up a certain algebraic exactness in order to ensure existence of non-trivial morphisms between non-trivial finite groups we define a new collection of Abelian groups. The following identifications are made under the assumption that $X$ satisfies either the $\mathsf{Homeo}$-PCVP or FPP restriction:
\[\mathcal{A}_{\widetilde{C}(X,X)_{\mathsf{Homeo}(X)}}=\{0\}\]
\begin{equation}\label{Agroups}
\mathcal{A}_{\widetilde{C}(X,X)^U_{(m,n)}}\cong\left\{\begin{array}{cc}
\mathbb{Z} & \mathrm{for}\;(m,n)=(\infty,\infty)\;\mathrm{and}\;\widetilde{C}(X,X)^U_{(m,n)}\neq\emptyset\\
\mathbb{Z}_2 & \mathrm{for}\;(m,n)\neq(\infty,\infty)\;\mathrm{and}\;\widetilde{C}(X,X)^U_{(m,n)}\neq\emptyset\\
\{0\} & \mathrm{if}\;\widetilde{C}(X,X)^U_{(m,n)}=\emptyset
\end{array}\right.\qquad
\end{equation}
If $I=(\mathbb{N}\times\mathbb{N})\cup\{(\infty,\infty)\}$, then we define an analogous counterpart to $\mathcal{S}_U$.
\begin{equation}\label{Amodule}
\mathcal{A}_U=\mathcal{A}_{\widetilde{C}(X,X)_{\mathsf{Homeo}(X)}}\oplus\bigoplus_{\substack{(m,n)\in I\\0<n\leq m}}\mathcal{A}_{\widetilde{C}(X,X)^U_{(m,n)}}=\bigoplus_{\substack{(m,n)\in I\\0<n\leq m}}\mathcal{A}_{\widetilde{C}(X,X)^U_{(m,n)}}
\end{equation}
For ease of notation, if there is no confusion, we shall denote the groups of \eqref{Agroups} by $\mathcal{A}^U_{m,n}$, where $(m,n)$ is the finite indice of $\widetilde{C}(X,X)^U_{(m,n)}$ or the infinite pair corresponding to $\widetilde{C}(X,X)^U_{(\infty,\infty)}$, as the context implies.

\begin{mylem}\label{lemBc1}
Let $X$ be a Hausdorff space with the necessary fixed point restrictions, and let $W\subseteq V\subseteq U\subseteq X$ be arbitrary open sets with inclusion maps $j:W\hookrightarrow V$ and $i:V\hookrightarrow U$. Let $\mathcal{F}:X\longrightarrow\mathbf{Mod}_{\mathbb{Z}}$ be the functor from $X$ to the category of $\mathbb{Z}$-modules defined according to
\[\mathcal{F}(U)=\mathcal{A}_{U}\qquad\mathcal{F}(i)=\rho^U_V:\mathcal{A}_U\longrightarrow\mathcal{A}_V\] 
Then $\mathcal{F}$ is a monopresheaf under the following homomorphism specifications
\[\rho^U_V(\mathcal{A}_U)=\bigoplus_{\substack{(m,n)\in(\mathbb{N}\times\mathbb{N})\cup\{(\infty,\infty)\}\\0<n\leq m}}\rho^U_V\left(\mathcal{A}^U_{m,n}\right)\quad\mathrm{where}\quad\rho^U_V=\left\{\begin{array}{cc}
\mathbf{1} & \mathrm{if}\quad\mathcal{A}^U_{m,n}\cong\mathcal{A}^V_{m,n}\\
\mathbf{0} & \mathrm{if}\quad\mathcal{A}^U_{m,n}\ncong\mathcal{A}^V_{m,n}
\end{array}\right.\]
\end{mylem}
\begin{proof}
Let us first prove that $\mathcal{F}$ is a presheaf, where without loss of generality we may suppose $f,g\notin\mathsf{Homeo}(X)$. By definition $\rho^U_U\equiv\mathbf{1}$, and so $F(Id_U)=\rho^U_U$ is the identity on $\mathcal{F}(U)$. Given non-empty open sets and inclusions $W\hookrightarrow V\hookrightarrow U$ and an arbitrary pair $(m,n)$, either all sets $\widetilde{C}(X,X)^*_{(m,n)}$ are trivial, or $\widetilde{C}(X,X)^V_{(m,n)}=\widetilde{C}(X,X)^W_{(m,n)}=\emptyset$, or only $\widetilde{C}(X,X)^W_{(m,n)}=\emptyset$. In the first case $\mathcal{A}^U_{m,n}=\mathcal{A}^V_{m,n}=\mathcal{A}^W_{m,n}\cong\{0\}$, thus it immediately follows that $\rho^U_W\equiv\rho^V_W\circ\rho^U_V$. In the second case $\mathcal{A}^U_{m,n}\cong\mathbb{Z}_2$ (or $\mathcal{A}^U_{m,n}\cong\mathbb{Z}$) but $\mathcal{A}^V_{m,n}=\mathcal{A}^W_{m,n}\cong\{0\}$, thus $\rho^U_W\equiv\mathbf{0}=\mathbf{0}\circ\mathbf{0}\equiv\rho^V_W\circ\rho^U_V$. Finally, in the latter case $\mathcal{A}^U_{m,n}=\mathcal{A}^V_{m,n}\cong\mathbb{Z}_2$ or ($\mathcal{A}^U_{m,n}=\mathcal{A}^V_{m,n}\cong\mathbb{Z}$) while $\mathcal{A}^W_{m,n}\cong\{0\}$, hence $\rho^U_W\equiv\mathbf{0}=\mathbf{0}\circ\mathbf{1}\equiv\rho^V_W\circ\rho^U_V$. Now if $U=\emptyset$ there is nothing to prove since all groups are trivial; likewise if just $W=V=\emptyset$, the fact that $\mathcal{A}^V_{m,n}=\mathcal{A}^W_{m,n}\cong\{0\}$ similarly forces $\rho^U_W,\rho^V_W$, and $\rho^U_V$ to all be the $\mathbf{0}$ map. If only $W=\emptyset$ we still must have $\rho^U_W=\rho^V_W\equiv\mathbf{0}$, so regardless of the structure of $\mathcal{A}^U_{m,n}$ and $\mathcal{A}^V_{m,n}$ we obtain $\rho^U_W=\rho^V_W\circ\rho^U_V$.

Let $\{U_\alpha\}_{\alpha\in\Lambda}$ be an open cover of $U$ and suppose there exist $a,b\in\mathcal{F}(U)$ such that $\rho^U_{U_\alpha}(a)=\rho^U_{U_\alpha}(b)$ for all $\alpha\in\Lambda$. Fixing a pair $(m_0,n_0)$, in the case that $\widetilde{C}(X,X)^U_{(m_0,n_0)}=\emptyset$, it follows that $\mathcal{A}^U_{m_0,n_0}=\mathcal{A}^{U_\alpha}_{m_0,n_0}\cong\{0\}$ for all $U_\alpha$, thus $a_{m_0,n_0}=b_{m_0,n_0}=0\in\mathcal{A}^{U_\alpha}_{m_0,n_0}$ for all $\alpha\in\Lambda$. On the other hand suppose there exists a pair $(m_1,n_1)$ such that $\widetilde{C}(X,X)^U_{(m_1,n_1)}\neq\emptyset$. Then there exists an element $U_1\in\{U_\alpha\}_{\alpha\in\Lambda}$ satisfying $\widetilde{C}(X,X)^{U_1}_{(m_1,n_1)}$ is non-empty. It follows that $\mathcal{A}^U_{m_1,n_1}=\mathcal{A}^{U_0}_{m_1,n_1}\cong\mathbb{Z}_2$ (or $\mathcal{A}^U_{m_1,n_1}=\mathcal{A}^{U_0}_{m_1,n_1}\cong\mathbb{Z}$), thus $\rho^U_{U_0}\equiv\mathbf{1}$ and $a_{m_1,n_1}=\rho^U_{U_\alpha}(a_{m_1,n_1})=\rho^U_{U_\alpha}(b_{m_1,n_1})=b_{m_1,n_1}$. These two arguments prove that for any arbitrary pair $(m,n)$-- whether finite or infinite-- the elements $a_{m,n}$ and $b_{m,n}$ are equal. The fact that $a=b$ is now immediate from the definition of $\mathcal{F}(U)$:
\[a=\sum_{\substack{(m,n)\in(\mathbb{N}\times\mathbb{N})\cup\{(\infty,\infty)\}\\0<n\leq m}}a_{m,n}\qquad\quad b=\sum_{\substack{(m,n)\in(\mathbb{N}\times\mathbb{N})\cup\{(\infty,\infty)\}\\0<n\leq m}}b_{m,n}\]
In the special case of $U=\emptyset$ the monopresheaf condition is trivially true since the only open cover is $U$ itself.
\end{proof}
Unfortunately, for $|X|\geq2$ and not given the trivial topology, $\mathcal{F}$ is not a sheaf since it fails the gluing condition in the case that $\widetilde{C}(X,X)^{U_\alpha}_{(m,n)}\neq\emptyset$ and $\widetilde{C}(X,X)^{U_\beta}_{(m,n)}\neq\emptyset$, but $\widetilde{C}(X,X)^{U_\alpha\cap U_\beta}_{(m,n)}=\emptyset$ for some $\alpha,\beta\in\Lambda$. For the basic facts concerning the sheafification $\widetilde{\mathcal{F}}$ of a monopresheaf $\mathcal{F}$ along with other general sheaf theoretic notions one can refer to any of a number of resources (e.g. \cite{GQ21},\cite{JR09},\cite{JPS55}). 

\begin{mydef}\label{defBc2}
For a presheaf $\mathcal{F}$ and $x\in X$, the stalk at $x$ is defined as the colimit
\[\mathcal{F}_x=\left(\varinjlim\limits_{U\ni x}\mathcal{F}(U),\;\rho^U_x:\mathcal{F}(U)\longrightarrow\varinjlim\limits_{U\ni x}\mathcal{F}(U)\right)\] 
where the collection of open subsets $\{U\}_{x\in U}$ forms a directed set under inclusion, and $\rho^U_x(a)=a_x$. Endowed with a suitable topology, the \'{e}space \'{e}tal\'{e} (or stalk space) is the triple 
\[\left(S\mathcal{F}=\bigsqcup_{x\in X}\mathcal{F}_x,X,\pi\right)\] 
where $\pi:S\mathcal{F}\longrightarrow X$ is the obvious projection map satisfying $\pi(a_x)=x$ for all $a_x\in\mathcal{F}_x$. If we denote by $\Lambda[S\mathcal{F},\pi]$ the union over all open $U\subseteq X$ of the following sets of sections \[\Lambda[S\mathcal{F},U,\pi]=\left\{\widetilde{s}:U\longrightarrow S\mathcal{F}|\;\widetilde{s}(x)=\rho^U_x(s)=s_x,\;s\in\mathcal{F}(U)\;\mathrm{and}\; x\in U\right\} \]
then $S\mathcal{F}$ is given the coarsest topology which makes every $\widetilde{s}\in\Lambda[S\mathcal{F},\pi]$ continuous. 
\end{mydef}
In particular, $\Omega\subseteq S\mathcal{F}$ is open if and only if for all sections $\widetilde{s}\in\Lambda[S\mathcal{F},\pi]$ the set \[U_{\widetilde{s}}=\{x\in U|\;\widetilde{s}(x)=s_x\in\Omega\}\] 
is open for every open subset $U\subseteq X$. This makes $\pi$ a local homeomorphism and provides $S\mathcal{F}$ with a basis topology of open sets $\{\widetilde{s}(U)|\;s\in\mathcal{F}(U),U\subseteq X\}$.

\begin{mydef}\label{defBc3}
Given a presheaf $\mathcal{F}$, the sheafification of $\mathcal{F}$ is the sheaf of all continuous sections $\widetilde{\mathcal{F}}=\Gamma[S\mathcal{F},\pi]$, where $\widetilde{\mathcal{F}}(U)=\{\sigma:U\longrightarrow S\mathcal{F} |\;\pi\circ\sigma=Id_U\}$\footnote{Following common notation, and to remove confusion if $X$ is viewed as an open subspace of some larger space, we denote the set of global sections $\widetilde{\mathcal{F}}(X)$ by $\Gamma(X,\widetilde{\mathcal{F}})$}. We will denote by $\eta:\mathcal{F}\longrightarrow\widetilde{\mathcal{F}}$ the presheaf map defined according to
\[\eta:\mathcal{F}(U)\longrightarrow\widetilde{\mathcal{F}}(U)\qquad(\eta(s))(x)=\widetilde{s}(x)=\rho^U_x(s)=\widetilde{s}_x\in\mathcal{F}_x,\;\forall x\in U\]
\end{mydef}

\begin{mylem}\label{lemBc2}
Let $X$ be a Hausdorff space (Homeo-PCVP or FPP). Then $X$ possesses the PCVP if and only if for every $s\in\mathcal{F}(X)$ the global section $\eta(s):X\longrightarrow S\mathcal{F}$ has torsion image. Conversely, $X$ is not a PCVP space if and only if there exists $s\in\mathcal{F}(X)$ such that $(\eta(s))(x)$ is of infinite order for all $x\in X$.
\end{mylem}
\begin{proof}
We first describe the stalk $\mathcal{F}_x$, using the usual equivalence class definition of directed limits of modules and the fact that colimits commute with direct sums:	
\[\varinjlim\limits_{U\ni x}\mathcal{A}_U=\bigoplus_{\substack{(m,n)\in(\mathbb{N}\times\mathbb{N})\cup\{(\infty,\infty)\}\\0<n\leq m}}\varinjlim\limits_{U\ni x}\left(\mathcal{A}^U_{m,n}\right)=\bigoplus_{\substack{(m,n)\in(\mathbb{N}\times\mathbb{N})\cup\{(\infty,\infty)\}\\0<n\leq m}}\left(\bigsqcup_{U\ni x}\mathcal{A}^U_{m,n}\right)/\sim\]
\[a^U_{m,n}\sim a^V_{m,n}\quad\mathrm{if\;and\;only\;if}\quad\rho^U_{U\cap V}(a^U_{m,n})=\rho^V_{U\cap V}(a^V_{m,n})\]
The above equalities imply that for every $U\ni x$ the $\rho^U_x$ maps of \Cref{defBc2} can be decomposed into the direct sum of summand maps $\rho^U_{x,m,n}$. By definition of directed limit, the inclusion maps $\rho^U_{x,m,n}$ fit into the commutative diagram
\[\begin{tikzcd} 
 & \varinjlim\limits_{U\ni x}\mathcal{A}^U_{m,n}   &  \\
\mathcal{A}^V_{m,n} \arrow[ur,"\rho^V_{x,m,n}"] \arrow[rr,"\rho^V_W"] &  & \mathcal{A}^W_{m,n} \arrow[ul,swap,"\rho^W_{x,m,n}"]
\end{tikzcd}\qquad W\hookrightarrow V\]
Assuming $(m_i,n_i)$ are finite it is not difficult to verify that if there exists a commuting pair $f,g:X\longrightarrow X$ such that $x\in\mathcal{CV}(f^{m_i},g^{n_i})$-- where $(m_i,n_i)$ are a minimal pair-- then $\mathcal{A}^U_{m_i,n_i}\cong\mathbb{Z}_2$ for all $U\in\{U\}_{x\in U}$; hence $\varinjlim_{(U\ni x)}\mathcal{A}^U_{m_i,n_i}\cong\mathbb{Z}_2$ and all inclusion maps $\rho^U_{x,m_i,n_i}$ are the identity. On the other hand, suppose that $x\notin\mathcal{CV}(f^{m_i},g^{n_i})$ for any pair $(f,g)$, or $(m_i,n_i)$ never form a minimal pair in this sense. Either there exists an open neighbourhood $N_x$ of $x$ such that $\widetilde{C}(X,X)^{N_x}_{(m_i,n_i)}$ is non-empty, or $\varinjlim_{(U\ni x)}\mathcal{A}^U_{m_i,n_i}\cong\mathcal{A}^U_{m_i,n_i}\cong\{0\}$ for all $U\in\{U\}_{x\in U}$, and all inclusion maps $\rho^U_{x,m_i,n_i}$ are trivial. Clearly, if $\widetilde{C}(X,X)^{N_x}_{(m_i,n_i)}\neq\emptyset$, then also $\widetilde{C}(X,X)^{X}_{(m_i,n_i)}\neq\emptyset$ and thus $\mathcal{A}^{N_x}_{m_i,n_i}\cong\mathcal{A}^{X}_{m_i,n_i}\cong\mathbb{Z}_2$. It follows from diagram chasing that $\mathcal{A}^{U_0}_{m_i,n_i}\cong\{0\}$ for some open $U_0\ni x$ forces triviality of $\varinjlim_{(U\ni x)}\mathcal{A}^U_{m_i,n_i}$ along with all inclusion maps $\rho^U_{x,m_i,n_i}$; otherwise, $\varinjlim_{(U\ni x)}\mathcal{A}^U_{m_i,n_i}\cong\mathbb{Z}_2$ and $\rho^X_{x,m_i,n_i}$ is the identity. 

Now, for $X$ a PCVP space either $\varinjlim_{(U\ni x)}\mathcal{A}^U_{\infty,\infty}\cong\mathbb{Z}$ or it is the trivial group; regardless, we have $\mathcal{A}^X_{\infty,\infty}=\{0\}$ and thus the trivial inclusion map $\rho^X_{x,\infty,\infty}$. If $X$ is not a PCVP space, then by definition $\widetilde{C}(X,X)^{*}_{(\infty,\infty)}\neq\emptyset$, and thus $\mathcal{A}^U_{\infty,\infty}\cong\mathbb{Z}$ for any $x\in X$ and for every $U\in\{U\}_{x\in U}$. It follows that $\varinjlim_{(U\ni x)}\mathcal{A}^U_{\infty,\infty}\cong\mathbb{Z}$ and all inclusion maps $\rho^U_{x,\infty,\infty}$-- including $\rho^X_{x,\infty,\infty}$-- are the identity. Combined with the above results we are now able to use the $\rho^X_x$ maps to investigate the structure of $\Gamma(X,\widetilde{\mathcal{F}})$, and hence determine how $X$ being a PCVP space influences this.  

It is clear from the definitions that the morphisms $\rho^X_x$ exist for any $x$, and $\mathcal{F}(X)$ is a non-trivial purely torsion module if $X$ possesses the PCVP. Moreover, we have shown that $\rho^X_{x,\infty,\infty}$ is always trivial and $\varinjlim\limits_{U\ni x}\mathcal{A}^U_{m,n}$ is finite for every finite $(m,n)$ pair. Hence for arbitrary $s\in\mathcal{F}(X)$ we have
\begin{equation}
(\eta(s))(x)=\rho^X_x(s)=\sum_{\substack{(m,n)\in(\mathbb{N}\times\mathbb{N})\\0<n\leq m}}\rho^X_{x,m,n}(s_{m,n})\in\bigoplus_{i\in I}\mathbb{Z}_2	
\end{equation}
Due to $\mathcal{F}$ being a monopresheaf, it follows from the injective property of $\eta$ that unless $s=0\in\mathcal{F}(X)$ there must exist a point $x$ such that $(\eta(s))(x)\in\mathcal{F}_x$ is a non-trivial element. Thus $\eta(\mathcal{F}(X))\subseteq\Gamma(X,\widetilde{\mathcal{F}})$ consists only of global sections which do not identically vanish on finitely many points, and take values which are all torsion. If $X$ is not a PCVP space we know that $\mathcal{F}(X)$ contains a copy of $\mathcal{A}^X_{\infty,\infty}=\mathbb{Z}$; since $\rho^X_{x,\infty,\infty}$ is the identity, taking $s\in\mathcal{F}(X)$ such that $s_{\infty,\infty}\neq0$ provides an element 
\begin{equation}
(\eta(s))(x)=\rho^X_x(s)=\left(\sum_{\substack{(m,n)\in(\mathbb{N}\times\mathbb{N})\\0<n\leq m}}\rho^X_{x,m,n}(s_{m,n}),s_{\infty,\infty}\right)\in\left(\bigoplus_{i\in I}\mathbb{Z}_2,\mathbb{Z}\right)	
\end{equation}
which is of infinite order. This is true for any $x$, so it follows that $\eta(\mathcal{F}(X))\subseteq\Gamma(X,\widetilde{\mathcal{F}})$ must contain at least one global section which takes only non-torsion values.
\end{proof}

\begin{mycor}\label{corBc1}
Let $X$ be a paracompact Hausdorff space with the necessary fixed point restrictions. Then the sheaf cohomology groups $H^p(X,\widetilde{\mathcal{F}})$-- or equivalently the C\v{e}ch cohomology groups $\check{H}^p(X,\widetilde{\mathcal{F}})$-- can detect PCVP structure.
\end{mycor}
\begin{proof}
It is a classical result of Godement \cite{RG58} that for $X$ paracompact and $\mathcal{F}$ any presheaf, there are isomorphisms 
\[\check{H}^p(X,\mathcal{F})\cong\check{H}^p(X,\widetilde{\mathcal{F}})\cong H^p(X,\widetilde{\mathcal{F}})\quad p\geq0\]
In particular, $\check{H}^0(X,\widetilde{\mathcal{F}})\cong\Gamma(X,\widetilde{\mathcal{F}})$, which recovers the information of \Cref{lemBc2}.
\end{proof}

\subsection{Some Cohomological Computations}\label{secB4}
This section is devoted to some elementary computations and implications of the sheaf-theoretic framework built in the previous section. \Cref{lemBc2} implies the necessity of knowing where torsion lies inside $H^0(X,\widetilde{\mathcal{F}})$, thus as matter of prime importance we would wish to ensure that homeomorphic spaces have isomorphic cohomological structure, not just isomorphic cohomology groups. This is answered in the affirmative by noting that if $\phi:X\longrightarrow Y$ is a homeomorphism, then $\phi$ induces one-to-one equivalences between the collection of open sets of each space, along with an equivalence $\widetilde{C}(Y,Y)\equiv\widetilde{C}(X,X)$. It immediately follows that for every commuting pair $(f_Y,g_Y)$ and open set $U_Y\subseteq Y$ there exists an associated triplet $(f_X,g_X,U_X)$ such that
\[S^{U_X}_{f_X,g_X}\cong S^{U_Y}_{f_Y,g_Y}\implies\mathcal{F}(U_X)=\mathcal{F}(U_Y)\]
By definition this leads to an isomorphism of stalks $\mathcal{F}_y\cong\mathcal{F}_{\phi(x)}$, and due to the obvious equivalence $\widetilde{s}(U_X)\leftrightarrow\widetilde{s}(U_Y)$ of basis sets we obtain homeomorphic associated stalk spaces. The presheaf map behaviour can thus be related through the isomorphism $\tilde{\phi}$ on stalk spaces induced by $\phi$
\[\tilde{\phi}(\eta_X(s)(x))=\eta_Y(s)(\phi(x))=\eta_Y(s)(y)\]
It is similarly easy to see that the sheaf construction behaves as expected with respect to the singleton set; indeed, the only self-map of $X=\{pt\}$ is $Id_X$, hence $H^0(\{pt\},\widetilde{\mathcal{F}})\cong\Gamma(\{pt\},\widetilde{\mathcal{F}})=\mathcal{F}(\{pt\})=\{0\}$. Providing the first explicit differentiating example between PCVP and WCVP, the following proposition concerning discrete finite spaces contrasts the negative result for WCVP obtained in \Cref{lemBa2}. While this is unsurprising from a combinatorial viewpoint, the essential fact is that we can give a proof which is almost purely cohomological.

\begin{mypro}\label{proBd1}
If $X$ is a discrete finite space, then $X$ possesses the PCVP.	
\end{mypro}
\begin{proof}
It is first necessary to ensure that $X=\{x_1,\ldots,x_n\}$ is at least $\mathsf{Homeo}$-PCVP, so suppose that $f,g:X\longrightarrow X$ commute for $f\in\mathsf{Homeo}(X)$. Since $f^{-1}g$ is well-defined and represents a permutation on its image, then viewing $f^{-1}g\in S_m$ its order is equal to $k$ for some $1\leq k\leq m$, where $|f^{-1}g(X)|=m$. By commutativity it follows that $(f^{-1}g)^k(x_i)=f^{-k}g^k(x_i)=x_i$ for all $x_i\in f^{-1}g(X)$, thus $x_i\in\mathcal{CV}(f^k,g^k)$. We now denote  $X_1=X\setminus\{x_1\}$ and consider the following section of the Mayer-Vietoris sequence in sheaf cohomology
{\footnotesize
\[\begin{tikzcd} 
0 \arrow[r] & H^0(X,\widetilde{\mathcal{F}}) \arrow[r] & H^0(\{x_1\},\widetilde{\mathcal{F}})\oplus H^0(X_1,\widetilde{\mathcal{F}}) \arrow[r] & H^0(X_1\cap\{x_1\},\widetilde{\mathcal{F}}) \arrow[r] &  H^1(X,\widetilde{\mathcal{F}})
\end{tikzcd}\]}
Since $X_1\cap\{x_1\}=\emptyset$ and $\Gamma(\widetilde{\mathcal{F}},\{pt\})=\{0\}$ this simplifies to a very short exact sequence
\[\begin{tikzcd} 
0 \arrow[r] & H^0(X,\widetilde{\mathcal{F}}) \arrow[r] &0\oplus H^0(X_1,\widetilde{\mathcal{F}}) \arrow[r] & 0 
\end{tikzcd}\]
and we thus obtain an isomorphism $\Gamma(X,\widetilde{\mathcal{F}})\cong\Gamma(X_1,\widetilde{\mathcal{F}})$. Recursively defining $X_0=X$ and $X_k=X_{k-1}\setminus\{x_k\}$ it immediately follows that $\Gamma(X_{k-1},\widetilde{\mathcal{F}})\cong\Gamma(X_k,\widetilde{\mathcal{F}})$ and thus $\Gamma(X,\widetilde{\mathcal{F}})\cong\Gamma(X_{n-1}=\{x_n\},\widetilde{\mathcal{F}})=\{0\}$.
\end{proof}

\begin{myrem}
The example of discrete spaces also furnishes proof that non-homeomorphic spaces can possess identical sheaf cohomology in the strongest sense. For less trivial results it will be useful to view $X$ as a ringed space $(X,\mathcal{O}_X)$, where we take the structure sheaf $\mathcal{O}_X$ to be the constant sheaf $\widetilde{\mathbb{Z}}_X$ obtained by sheafifying the constant presheaf $\mathbb{Z}_X$. This identification allows us to view the sheaf $\widetilde{\mathcal{F}}$ of $\mathbb{Z}$-modules as an $\mathcal{O}_X$-module and thus speak about quasi-coherency/coherency, maps between ringed spaces, pushforward and pullback sheaves, and other important notions. For now we will merely adopt the language without belaboring the underlying technical points. 	
\end{myrem}

The above proposition immediately brings to mind the first major computational consideration: a Mayer-Vietoris argument which preserves the presheaf information. At a minimum we would desire the ability to compute $H^0(X,\widetilde{\mathcal{F}})$ through decomposition of $X$ into the union of subspaces for which the cohomology is already known, and which need not have trivial pairwise intersection. Unfortunately, while sheaves are only defined over open subsets, $\mathcal{F}$ is only well-defined over $\mathsf{Homeo}$-PCVP and FPP spaces. However, all known ``nice enough" examples of such spaces are compact \cite[Theorem 3]{EC59}, hence for $X$ connected-- such as an adjunction space or non-trivial fibre bundle-- we cannot expect to find an open decomposition. We will introduce two notions which provide first steps by which these problems can be rectified. 

\begin{mydef}\label{defBd1}
If $\overline{U}$ denotes the closure of $U$, then for any open set $U\subseteq X$	\[\mathcal{A}_{\widetilde{C}(X,X)^{\overline{U}}_{(m,n)}}\cong\left\{\begin{array}{cc}
\mathbb{Z} & \mathrm{for}\;(m,n)=(\infty,\infty)\;\mathrm{and}\;\widetilde{C}(X,X)^{\overline{U}}_{(m,n)}\neq\emptyset\\
\mathbb{Z}_2 & \mathrm{for}\;(m,n)\neq(\infty,\infty)\;\mathrm{and}\;\widetilde{C}(X,X)^{\overline{U}}_{(m,n)}\neq\emptyset\\
\{0\} & \mathrm{if}\;\widetilde{C}(X,X)^{\overline{U}}_{(m,n)}=\emptyset
\end{array}\right.\]
The monopresheaf $\mathcal{G}$ is defined analogously to the $\mathcal{F}$ of \Cref{lemBc1}, where $\mathcal{G}(U)=\mathcal{A}_{\overline{U}}$.
\end{mydef}
A central importance of this definition is that it allows us to define $\mathcal{G}$ on the collection of closed sets which arise as closures of open subsets, namely $\mathcal{G}(\overline{U})=\mathcal{G}(U)$. It is also easy to see that $\mathcal{G}(U)=\mathcal{F}(U)$ whenever $U=\overline{U}$; in particular, $\mathcal{G}(X)=\mathcal{F}(X)$. What is not immediate is whether this equality is preserved under applying the sheafification operation. In particular, we wish to show that \Cref{lemBc2} holds when $\mathcal{G}$ replaces $\mathcal{F}$.

\begin{mylem}\label{lemBd2}
The monopresheaf $\mathcal{G}$ satisfies the conditions of \Cref{lemBc2}; moreover if $X$ is not a PCVP space, then $\eta_{\mathcal{F}}(\mathcal{F}(X))\cong\eta_{\mathcal{G}}(\mathcal{G}(X))$ 
\end{mylem}
\begin{proof}
In the case that $X$ is a PCVP space, a careful reworking of the proof of \Cref{lemBc2} analogously proves the desired result for $\mathcal{G}$. To prove the stronger condition when $X$ not a PCVP space we first show that a non-trivial map of presheaves $\varphi:\mathcal{F}\longrightarrow\mathcal{G}$ exists. By the direct summand description of $\mathcal{F}(U)=\mathcal{A}_U$ and $\mathcal{G}(U)=\mathcal{A}_{\overline{U}}$, it suffices to prove existence of the following commutative diagram for both $(m,n)=(\infty,\infty)$ and $(m,n)$ a finite pair.	
\begin{equation}
\begin{tikzcd} 
\mathcal{A}^U_{m,n} \arrow{r}{\varphi_U} \arrow[d,swap,"(\rho_\mathcal{F})^U_V"] & \mathcal{A}^{\overline{U}}_{m,n} \arrow{d}{(\rho_\mathcal{G})^U_V}\\
\mathcal{A}^V_{m,n} \arrow{r}{\varphi_V} & \mathcal{A}^{\overline{V}}_{m,n}
\end{tikzcd}
\end{equation}
To begin with, let $\varphi_U$ be the zero morphism whenever $(m,n)\neq(\infty,\infty)$; this is both a matter of technical necessity in order to deal with the problematic case where only the lower left group is trivial, as well as a technical convenience since we only care about preserving non-torsion information. Similarly, since we are interested specifically in the global sections, we define
\[\varphi_U=\left\{\begin{array}{cc}
\mathbf{1} & U=X,(m,n)=(\infty,\infty)\\
\mathbf{0} & \mathrm{otherwise}
\end{array}\right.\] 
for which the commutative properties of the above diagram can be easily verified. For $X$ not a PCVP space we thus have a non-trivial morphism of presheaves $\varphi:\mathcal{F}\longrightarrow\mathcal{G}$; using the fact that the presheaf map $\eta$ of \Cref{defBc3} is natural with respect to presheaves, the following diagrams commute:
\begin{equation}
\begin{tikzcd} 
\mathcal{F} \arrow{r}{\varphi} \arrow[d,swap,"\eta_{\mathcal{F}}"] & \mathcal{G}\arrow{d}{\eta_{\mathcal{G}}}\\
\widetilde{\mathcal{F}} \arrow{r}{\widetilde{\varphi}} & \widetilde{\mathcal{G}}
\end{tikzcd}\implies
\begin{tikzcd} 
\mathcal{F}(X) \arrow{r}{\varphi_X} \arrow[d,swap,"\eta_{\mathcal{F}}"] & \mathcal{G}(X)\arrow{d}{\eta_{\mathcal{G}}}\\
\Gamma(X,\widetilde{\mathcal{F}}) \arrow{r}{\widetilde{\varphi}_X} & \Gamma(X,\widetilde{\mathcal{G}})
\end{tikzcd}
\end{equation}
Since $\varphi_X$ is the identity map the right-side diagram implies the desired isomorphism of submodules.
\end{proof}

Now that we have established a well-behaved presheaf $\mathcal{G}$, it is necessary to consider the open subsets $U$ which form desirable elements of a cover. We will indicate the presheaf (or sheaf) on $\overline{U}$ by $\mathcal{G}_{\overline{U}}$ whenever it is necessary to emphasize that $\overline{U}$ is not being viewed as a subspace of $X$.
\begin{mydef}\label{defBd2}
Let $X$ possess either $\mathsf{Homeo}$-PCVP or the FPP and $\iota:Y\hookrightarrow X$ be the inclusion of an open subspace. We call $Y$ an \textit{\textbf{admissible subspace}} if its closure satisfies the following conditions:
\begin{enumerate}
\item $\overline{Y}$ is either a $\mathsf{Homeo}$-PCVP or an FPP space
\item There exists a continuous map $\phi:X\longrightarrow\overline{Y}$ such that $\phi|_{\overline{Y}}\equiv Id_{\overline{Y}}$
\item If $\widetilde{C}(X,X)^{\overline{Y}}_{(\infty,\infty)}\neq\emptyset$ then $\widetilde{C}(\overline{Y},\overline{Y})^{\overline{Y}}_{(\infty,\infty)}\neq\emptyset$
\end{enumerate}
$X$ is said to admit an \textit{\textbf{admissible cover}} $\mathscr{U}$ if every $U\in\mathscr{U}$ is an admissible subspace, all intersections and unions possess the $\mathsf{Homeo}$-PCVP or the FPP, and the closure of intersections is the intersection of closures.
\end{mydef}

\begin{myrem}
Note that the final condition concerning admissible covers always holds for unions, and $\overline{U_i}\cap\overline{U_j}=\overline{U_i\cap U_j}$ ensures that $\mathcal{G}(U_i\cap U_j)=\mathcal{G}(\overline{U_i\cap U_j})=\mathcal{G}(\overline{U_i}\cap\overline{U_j})$. For $X$ paracompact these relations are a technical necessity for ensuring that the Mayer-Vietoris long exact sequence carries through with regards to closures.
\end{myrem}

\begin{mylem}\label{lemBd3}
Let $X$ be paracompact Hausdorff and $\iota:Y\hookrightarrow X$ be the inclusion of an admissible subspace. Then there exist induced inclusion and projection homomorphisms
\[\iota_*:\mathcal{G}_{\overline{Y}}(\overline{Y})\longrightarrow\mathcal{G}(Y)\qquad\pi:\mathcal{G}(Y)\longrightarrow\mathcal{G}_{\overline{Y}}(\overline{Y})\]
which preserve all non-torsion elements. In addition, there exist homomorphisms 
\[\widetilde{\varphi}_{\iota}:\Gamma(\overline{Y},\widetilde{\mathcal{G}})\longrightarrow\widetilde{\mathcal{G}}(Y)\qquad\widetilde{\varphi}_\pi:\widetilde{\mathcal{G}}(Y)\longrightarrow\Gamma(\overline{Y},\widetilde{\mathcal{G}})\]
satisfying $\widetilde{\varphi}_{\iota}(\eta_{\overline{Y}}(\mathcal{G}_{\overline{Y}}(\overline{Y})))=\eta(\iota_*(\mathcal{G}_{\overline{Y}}(\overline{Y})))$ and $\widetilde{\varphi}_\pi(\eta(\mathcal{G}(Y)))=\eta_{\overline{Y}}(\pi(\mathcal{G}(Y)))$.
\end{mylem}
\begin{proof}
Throughout, we take $(m,n)$ to be a finite pair. Given that $(f,g)\in\widetilde{C}(\overline{Y},\overline{Y})^{\overline{Y}}_{(m,n)}$ the map $\phi$ of \Cref{defBd2} provides existence of non-homeomorphic commuting maps $\iota f\phi,\iota g\phi:X\longrightarrow X$. It follows that there exists $y\in\mathcal{CV}(f^m,g^n)\cap\mathcal{CV}(\iota f^m\phi,\iota g^n\phi)$ and so $(\iota f\phi,\iota g\phi)\in\widetilde{C}(X,X)^{\overline{Y}}_{(m,n)}$. On the other hand, regardless of whether $\widetilde{C}(X,X)^{\overline{Y}}_{(m,n)}$ is empty or not, by definition $\mathcal{A}_{\widetilde{C}(\overline{Y},\overline{Y})^{\overline{Y}}_{(m,n)}}$ is either trivial or isomorphic to $\mathbb{Z}_2$. Turning to the case of the $(\infty,\infty)$ indice, the third admissibility condition assures that $\widetilde{C}(X,X)^{\overline{Y}}_{(\infty,\infty)}\neq\emptyset\implies\widetilde{C}(\overline{Y},\overline{Y})^{\overline{Y}}_{(\infty,\infty)}\neq\emptyset$; conversely, if $(f,g)\in\widetilde{C}(\overline{Y},\overline{Y})^{\overline{Y}}_{(\infty,\infty)}$ it is again easy to see that $(\iota f\phi,\iota g\phi)\in\widetilde{C}(X,X)^{\overline{Y}}_{(\infty,\infty)}$. Altogether, we have an isomorphism $\mathcal{A}_{\widetilde{C}(X,X)^{\overline{Y}}_{(\infty,\infty)}}\cong\mathcal{A}_{\widetilde{C}(\overline{Y},\overline{Y})^{\overline{Y}}_{(\infty,\infty)}}$ along with inclusions $\mathcal{A}_{\widetilde{C}(X,X)^{\overline{Y}}_{(m,n)}}\hookrightarrow\mathcal{A}_{\widetilde{C}(\overline{Y},\overline{Y})^{\overline{Y}}_{(m,n)}}$. Taking $\iota_*$ and $\pi$ to be the obvious inclusion and projection maps on the direct sum provides the presheaf result.

Since $\mathcal{G}$ is a monopresheaf we know that the presheaf map $\eta:\mathcal{G}\longrightarrow\widetilde{\mathcal{G}}$ is injective, hence isomorphic onto its image. By the definition of $\iota_*$ there exist submodule relations
\[\eta_{\overline{Y}}(\mathcal{G}_{\overline{Y}}(\overline{Y}))\cong\eta(\iota_*(\mathcal{G}_{\overline{Y}}(\overline{Y})))\subseteq\eta(\mathcal{G}(Y))\subseteq\widetilde{\mathcal{G}}(Y)\]
and we take $\widetilde{\varphi}_{\iota}$ to be the composition of the surjective homomorphism $q_1:\Gamma(\overline{Y},\widetilde{\mathcal{G}})\longrightarrow\eta_{\overline{Y}}(\mathcal{G}_{\overline{Y}}(\overline{Y}))$ with the isomorphism $\eta_{\overline{Y}}(\mathcal{G}_{\overline{Y}}(\overline{Y}))\xrightarrow{\cong}\eta(\iota_*(\mathcal{G}_{\overline{Y}}(\overline{Y})))$. To establish the corresponding result for $\widetilde{\varphi}_\pi$, we denote by $\widetilde{\pi}:\eta(\mathcal{G}(Y))\longrightarrow\eta(\iota_*(\mathcal{G}_{\overline{Y}}(\overline{Y})))$ the projection map induced by $\iota\pi$ and by $q_2:\widetilde{\mathcal{G}}(Y)\longrightarrow\eta(\mathcal{G}(Y))$, the surjective homomorphism. Thus we define $\widetilde{\varphi}_\pi$ to be equal to the following composition of homomorphisms.
\[\widetilde{\mathcal{G}}(Y)\xrightarrow{q_2}\eta(\mathcal{G}(Y))\xrightarrow{\widetilde{\pi}}\eta(\iota_*(\mathcal{G}_{\overline{Y}}(\overline{Y})))\xrightarrow{\cong}\eta_{\overline{Y}}(\mathcal{G}_{\overline{Y}}(\overline{Y}))\xrightarrow{\cong}\eta_{\overline{Y}}(\pi(\mathcal{G}(Y)))\subseteq\Gamma(\overline{Y},\widetilde{\mathcal{G}})\]
\end{proof}

\begin{mythm}\label{thmBd1}
Suppose that $(X,\widetilde{\mathbb{Z}}_X)$ is a ringed, paracompact Hausdorff space which admits an admissible cover $\mathscr{U}$. If $\mathscr{U}=\{U_1,U_2\}$, then $X$ is a PCVP space if both subspaces are. If for all elements of $\mathscr{U}=\{U_1,\ldots,U_n\}$ pairwise intersections are trivial, then $X$ is a PCVP space if all elements of $\mathscr{U}$ are.
\end{mythm}
\begin{proof}
Since $\widetilde{\mathcal{G}}$ takes values in $\mathbf{Mod}_{\mathbb{Z}}$, which is an Abelian category, the property of being a sheaf is equivalent to existence of an exact sequence
\begin{equation}\label{sheafexactseq}
\begin{gathered}
\begin{tikzcd} 
0 \arrow[r] & \widetilde{\mathcal{G}}(U) \arrow[r,"\widetilde{\varphi}_1"] & \prod_{\alpha\in\Lambda}\widetilde{\mathcal{G}}(U_\alpha)\arrow[r,"\widetilde{\varphi}_2"] & \prod_{(\alpha,\beta)\in\Lambda\times\Lambda}\widetilde{\mathcal{G}}(U_\alpha\cap U_\beta) 
\end{tikzcd}\\
\widetilde{\varphi}_1=(\widetilde{\rho}^U_{U_\alpha})_{\alpha\in\Lambda}\quad\mathrm{and}\quad\widetilde{\varphi}_2=\left(\widetilde{\rho}^{U_\alpha}_{U_\alpha\cap U_\beta}-\widetilde{\rho}^{U_\beta}_{U_\alpha\cap U_\beta}\right)_{(\alpha,\beta)\in\Lambda\times\Lambda}
\end{gathered}
\end{equation}
for any open cover $\{U_\alpha\}_{\alpha\in\Lambda}$ of an open subset $U\subseteq X$. Moreover, by the construction in \Cref{defBd1} it is not difficult to verify that for any such arbitrary open cover the presheaf restriction maps induce an inclusion $\varphi_1:\mathcal{G}(U)\longrightarrow\bigoplus_{\alpha\in\Lambda}\mathcal{G}(U_\alpha)$. Now, an admissible cover $\mathscr{U}=\{U_1,U_2\}$ of $X$ also gives rise to the following section of the standard Mayer-Vietoris sequence, where all groups are well-defined due to the admissibility condition:
\begin{equation*}
\begin{tikzcd} 
0 \arrow[r] & \Gamma(X,\widetilde{\mathcal{G}}) \arrow[r] & \Gamma(\overline{U}_1,\widetilde{\mathcal{G}})\oplus \Gamma(\overline{U}_2,\widetilde{\mathcal{G}}) \arrow[r] & \Gamma(\overline{U}_1\cap \overline{U}_2,\widetilde{\mathcal{G}}) \arrow[r] &  H^1(X,\widetilde{\mathcal{G}})
\end{tikzcd}	
\end{equation*}
Putting all the pieces together with respect to our admissible cover $\mathscr{U}=\{U_1,U_2\}$ we obtain the following diagram, where the top row is the Mayer-Vietoris exact sequence, the middle row exact sequence comes from \eqref{sheafexactseq}, the third row has injective $\varphi_1$ but is not necessarily exact, and the middle vertical pair of homomorphisms are as defined in \Cref{lemBd3}:
\begin{equation}
\begin{tikzcd} 	
0 \arrow[r] & \Gamma(X,\widetilde{\mathcal{G}}) \arrow[r] & \Gamma(\overline{U}_1,\widetilde{\mathcal{G}})\oplus \Gamma(\overline{U}_2,\widetilde{\mathcal{G}}) \arrow[r] & \Gamma(\overline{U}_1\cap \overline{U}_2,\widetilde{\mathcal{G}})\\
0 \arrow[r] & \Gamma(X,\widetilde{\mathcal{G}}) \arrow[r,"\widetilde{\varphi}_1"] \arrow[u,equal]& \widetilde{\mathcal{G}}(U_1)\oplus\widetilde{\mathcal{G}}(U_2)\arrow[r,"\widetilde{\varphi}_2"] \arrow[u,transform canvas={xshift=5ex},swap,"\widetilde{\varphi}_{\pi,U_2}"]\arrow[u,transform canvas={xshift=-5ex},"\widetilde{\varphi}_{\pi,U_1}"]& \widetilde{\mathcal{G}}(U_1\cap U_2)\\
0 \arrow[r] & \mathcal{G}(X) \arrow[r,"\varphi_1"]\arrow[u,"\eta"] & \mathcal{G}(U_1)\oplus\mathcal{G}(U_2)\arrow[r]\arrow[u,"\eta"] & \mathcal{G}(U_1\cap U_2) 
\end{tikzcd}
\end{equation}
Assume to the contrary that $\overline{U}_1$ and $\overline{U}_2$ are PCVP spaces but $X$ is not. By \Cref{lemBd2}, $\eta(\mathcal{G}(X))$ contains a strictly non-torsion section-- in fact, the proof of \Cref{lemBc2} shows that there exist countably many such sections. Since $\widetilde{\varphi}_1$ and $\varphi_1$ are restriction maps, the lower left-hand square commutes due to $\eta$ being a morphism of presheaves. Traversing the diagram from $\mathcal{G}(X)$ to the top row along the center column we obtain the module $L=(\widetilde{\varphi}_{\pi,U_1}\oplus\widetilde{\varphi}_{\pi,2})(\eta(\varphi_1(\mathcal{G}(X))))$. Since $\eta$ and $\varphi_1$ are injective, and $\widetilde{\varphi}_{\pi,U_i}$ sends non-torsion sections to non-torsion sections \footnote{By this we specifically mean that if $(\eta(\varphi_1(s)))(u_1,u_2)\in\mathcal{G}_{u_1}\oplus\mathcal{G}_{u_2}$ is of infinite order, then so is its image under the stalk space homomorphism induced by $(\widetilde{\varphi}_{\pi,U_1}\oplus\widetilde{\varphi}_{\pi,U_2})$.}, it follows that there exists $\sigma\in L$ such that $\sigma(u_1,u_2)$ takes only values of infinite order. However, it is also seen that
\begin{equation}
\begin{aligned}
L\subseteq(\widetilde{\varphi}_{\pi,U_1}\oplus\widetilde{\varphi}_{\pi,2})(\eta(\mathcal{G}(U_1))\oplus\eta(\mathcal{G}(U_2)))&=\bigoplus_{i=1}^2\eta_{\overline{U_i}}(\pi(\mathcal{G}(\overline{U_i})))\\
&=\bigoplus_{i=1}^2\eta_{\overline{U_i}}(\mathcal{G}_{\overline{U_i}}(\overline{U_i}))\subseteq\Gamma(\overline{U_1},\widetilde{\mathcal{G}})\oplus \Gamma(\overline{U_2},\widetilde{\mathcal{G}})
\end{aligned}
\end{equation}
hence by \Cref{lemBd2} the PCVP condition implies that the only non-trivial sections $\sigma\in\eta_{\overline{U_i}}(\mathcal{G}_{\overline{U_i}}(\overline{U_i}))$ are such that $\sigma(u_1,u_2)$ is torsion for all $u_i\in U_i$.

Now suppose that $\mathscr{U}=\{U_1,\ldots,U_n\}$ is an admissible cover where $U_i\cap U_j=\emptyset$ for $i\neq j$. Pairwise disjointedness implies that every $U_i$ is closed, hence $U_i=\overline{U_i}$ and by admissibility we know that $V_k=\bigcup_{i=1}^kU_i$ is always a $\mathsf{Homeo}$-PCVP space for $1\leq k\leq n$. Since $V_{n-1}\cap U_n=\emptyset$ and $X=V_{n-1}\sqcup U_n$, the relevant section of the Mayer-Vietoris sequence is simply
\[\begin{tikzcd} 
0 \arrow[r] & \Gamma(X,\widetilde{\mathcal{G}}) \arrow[r] & \Gamma(V_{n-1},\widetilde{\mathcal{G}})\oplus\Gamma(U_n,\widetilde{\mathcal{G}}) \arrow[r] & 0 \arrow[r] &  H^1(X,\widetilde{\mathcal{G}})
\end{tikzcd}\]
The isomorphism $\Gamma(X,\widetilde{\mathcal{G}})\cong\Gamma(V_{n-1},\widetilde{\mathcal{G}})\oplus\Gamma(U_n,\widetilde{\mathcal{G}})$ has an obvious extension to $\Gamma(X,\widetilde{\mathcal{G}})\cong\bigoplus_{i=1}^n\Gamma(U_i,\widetilde{\mathcal{G}})$ by using the decomposition $V_{k}=U_{k}\sqcup V_{k-1}$ and repeating the Mayer-Vietoris argument $n-1$ times. The desired result is then an immediate consequence of applying the argument for the two-element admissible cover to the following diagram
\begin{equation}
\begin{tikzcd} 	
0 \arrow[r] & \Gamma(X,\widetilde{\mathcal{G}}) \arrow[r] & \bigoplus_{i=1}^n\Gamma(U_i,\widetilde{\mathcal{G}}) \arrow[r] & 0\\
0 \arrow[r] & \Gamma(X,\widetilde{\mathcal{G}}) \arrow[r,"\widetilde{\varphi}_1"] \arrow[u,equal]& \bigoplus_{i=1}^n\widetilde{\mathcal{G}}(U_i)\arrow[r,"\widetilde{\varphi}_2"] \arrow[u,"\oplus_i\widetilde{\varphi}_{\pi,U_i}"]& 0\\
0 \arrow[r] & \mathcal{G}(X) \arrow[r,"\varphi_1"]\arrow[u,"\eta"] & \bigoplus_{i=1}^n\mathcal{G}(U_i)\arrow[r]\arrow[u,"\eta"] & 0 
\end{tikzcd}
\end{equation}
\end{proof}

\begin{mypro}\label{proBd2}
Let $(X\times F,\widetilde{\mathbb{Z}}_{X\times F})$ be a ringed, connected and compact PCVP space-- necessarily Hausdorff-- for $F$ a discrete finite space, then $X\times F$ possesses the PCVP.	
\end{mypro}
\begin{proof}
Note that $X\times F=\bigsqcup_{i=1}^nX\times\{y_i\}=\bigsqcup_{i=1}^nX_i$ where $X_i\simeq X$ and $X_i$ is both open and closed. Compactness of $X$ along with finiteness of $F$ ensures that $X\times F$ is paracompact, while the $\mathsf{Homeo}$-PCVP nature of $X\times F$ is proven by applying an argument analogous to that of \Cref{proBd1} to the cover $\mathscr{U}$ of connected components. By \Cref{thmBd1} we only need to prove that $\mathscr{U}=\{X_1,\ldots,X_n\}$ is an admissible cover, and since all  $X_i$ are homeomorphic it suffices to show that $X_1$ is an admissible subspace.

The first condition of \Cref{defBd2} is trivially satisfied, and the second condition holds with $\phi:X\times F\longrightarrow X_1$ defined according to $\phi(x,y_i)=(x,y_1)$. Now assume that $(f,g)\in\widetilde{C}(X\times F,X\times F)^{X_1}_{(\infty,\infty)}$, where disjointedness of the cover together with connectedness of $X$ implies that $f:X_1\longrightarrow X_k$ and $g:X_1\longrightarrow X_j$. The pair of maps $\phi f\phi,\phi g\phi:X_1\longrightarrow X_1$ commute since $f$ and $g$ do, thus by the PCVP hypothesis there exists $(z,y_1)\in\mathcal{CV}((\phi f\phi)^m,(\phi g\phi)^n)=\mathcal{CV}(\phi f^m\phi,\phi g^n\phi)$. However, this means that $(z,y_1)\in\mathcal{CV}(f^m,g^n)\cap X_1$, which is a contradiction. It follows that $\widetilde{C}(X\times F,X\times F)^{X_1}_{(\infty,\infty)}$ must be empty and so the third admissibility condition is vacuously satisfied.
\end{proof}

Since sheaf behaviour is local, paralleling the discussion following \Cref{lemBa4} we might expect the local triviality condition of fibre bundles $(E,X,\pi)$ to allow for the choosing of a suitable open cover such that \Cref{proBd2} extends. In particular, we know that for a topological space $X$ the C\v{e}ch cohomology group $\check{H}^0(X,\widetilde{\mathcal{F}})$ is isomorphic to $\check{H}^0(\mathscr{U},\widetilde{\mathcal{F}})$ for any choice of open cover $\mathscr{U}$.

\begin{mycon}
Let $\Gamma$ be a finite group and $(E,X,\pi)$ be a principal $\Gamma$-bundle, where $(X,\widetilde{\mathbb{Z}}_X)$ is a ringed, connected, compact, Hausdorff, and PCVP space. If $X$ admits a covering of admissible subspaces, then the total space $E$ has the PCVP.
\end{mycon}

We now briefly turn to the second major computational tool we would wish to have at our disposal: the ability to determine the cohomology of a paracompact space $X\times Y$ from the cohomology of the factor spaces. As a simple example, since the admissible cover of \Cref{proBd2} is comprised of clopen subspaces, the argument holds equally well for $\widetilde{\mathcal{F}}$, and the proof of \Cref{thmBd1} shows that $H^0(X\times F,\widetilde{\mathcal{F}})\cong\bigoplus_{i=1}^n H^0(X,\widetilde{\mathcal{F}})$. This can be expressed in a manner which clearly emphasizes the role of $X$ and $F$. For any open subset $U\subseteq X\times F$ and any pair of sheaves $\widetilde{\mathcal{F}_1},\widetilde{\mathcal{F}_2}$ defined on $X\times F$ denote by $\widetilde{\mathcal{F}_1}\otimes_{\widetilde{\mathbb{Z}}_{X\times F}}\widetilde{\mathcal{F}_2}$ the sheafification of the presheaf tensor product $\widetilde{\mathcal{F}_1}\otimes_{p,\widetilde{\mathbb{Z}}_{X\times F}}\widetilde{\mathcal{F}_2}$ defined according to
\[(\widetilde{\mathcal{F}_1}\otimes_{p,\widetilde{\mathbb{Z}}_{X\times F}}\widetilde{\mathcal{F}_2})(U)=\widetilde{\mathcal{F}_1}(U)\otimes_{\widetilde{\mathbb{Z}}_{X\times F}(U)}\widetilde{\mathcal{F}_2}(U)\]
Now if $\widetilde{\mathcal{F}_1}$ and $\widetilde{\mathcal{F}_2}$ are sheaves defined on $X$ and $F$, respectively, and $p_1:X\times F\longrightarrow X$ and $p_2:X\times F\longrightarrow F$ are the projection maps, then recall that the pullback (inverse image) sheaf $p_i^{-1}\widetilde{\mathcal{F}}_i$ on $X\times F$ is the sheafification of the presheaf $p_i^+\widetilde{\mathcal{F}}_i$ defined as
\[p_i^+\widetilde{\mathcal{F}}_i(U):=\varinjlim\limits_{V\supseteq p_i(U)}\widetilde{\mathcal{F}}_i(V)\]
Considering all spaces as ringed spaces with structure sheaf $\widetilde{\mathbb{Z}}$ and using the fact that the pullback of a constant sheaf is itself a constant sheaf, we obtain a relation similar to the familiar K\"{u}nneth isomorphism concerning quasicoherent sheaves on schemes: 
\begin{equation}
\begin{aligned}
H^0(X\times F, \widetilde{\mathcal{F}})&\cong	
H^0(X\times F, (p_1^{-1}\widetilde{\mathcal{F}}_X)\otimes_{\widetilde{\mathbb{Z}}_{X\times F}} (p_2^{-1}\widetilde{\mathbb{Z}}_F))\\
&\cong H^0(X,\widetilde{\mathcal{F}})\otimes_{\mathbb{Z}} H^0(F,\widetilde{\mathbb{Z}}_F)\cong H^0(X,\widetilde{\mathcal{F}})\otimes_{\mathbb{Z}}\mathbb{Z}^n\cong\bigoplus_{i=1}^n H^0(X,\widetilde{\mathcal{F}}).		
\end{aligned}
\end{equation}
Due to $\widetilde{\mathcal{F}}$ not being a constant sheaf it is almost certain that such an elementary relation does not hold when $F$ is replaced by a non-discrete space, the subtlety of the algebraic topological hypotheses being noticeable even in shifting to the realm of locally constant sheaves (e.g. \cite[Theorem 1.7]{RG06}). Indeed, since some sort of coherency/finiteness assumption is frequently necessary for K\"{u}nneth theorems involving non-constant sheaves, one of the main purposes of the next result is to be able to introduce a rigourous framework by which the conjecture that follows can be attacked. 

\begin{mycon}\label{conBd1}
Whether $\mathbf{D}$ possesses the PCVP can be determined entirely by the sheaf cohomology groups $H^p([0,1],\widetilde{\mathcal{F}})$, given that one can prove a K\"{u}nneth theorem result for the product representation $\mathbf{D}\simeq [0,1]\times [0,1]$
\end{mycon}  

\begin{mylem}\label{lemBd1}
Let $(X,\widetilde{\mathbb{Z}}_X)$ be a ringed space which is Hausdorff and satisfies a $\mathsf{Homeo}$-PCVP or FPP condition. Then $\widetilde{\mathcal{F}}$ is a quasi-coherent sheaf on $(X,\widetilde{\mathbb{Z}}_X)$. 
\end{mylem}
\begin{proof}
By definition of $\mathcal{F}$ and the constant presheaf $\mathbb{Z}_X$, for any open subset $U\subseteq X$ 
\[\mathbb{Z}_X(U)=\mathbb{Z}\qquad\mathrm{and}\qquad\mathcal{F}(U)\subseteq\mathbb{Z}\oplus\left(\bigoplus_{i\in I}\mathbb{Z}_2\right)\]
where $I=\{(m,n)\in\mathbb{N}\times\mathbb{N}:\;0<n\leq m\}$. For each $j\in J=I\cup\{0,0\}$ define a pair of morphisms $(\varphi_j,\phi_j)$ which fit as needed into the following exact sequences in the obvious way
\[\begin{tikzcd} 
\mathbb{Z} \arrow[r,"\varphi_j\equiv(\times0)"]& \mathbb{Z} \arrow[r,"\phi_j\equiv(\times1)"] &\mathbb{Z}
\end{tikzcd}\qquad
\begin{tikzcd} 
\mathbb{Z} \arrow[r,"\varphi_j\equiv(\times2)"]& \mathbb{Z} \arrow[r,"\phi_j"] &\mathbb{Z}_2
\end{tikzcd}\qquad
\begin{tikzcd} 
\mathbb{Z} \arrow[r,"\varphi_j\equiv(\times1)"]& \mathbb{Z} \arrow[r,"\phi_j"] & 0
\end{tikzcd}\]
Extending on the right by 0, it is clear that these sequences remain exact, thus for each open $U$ we have the following exact sequence of modules which thus implies an exact sequence of presheaves
\begin{equation}
\begin{gathered}
\begin{tikzcd} 
\bigoplus_{j\in J}\mathbb{Z}_X(U) \arrow[r,"\varphi_U"]& \bigoplus_{j\in J}\mathbb{Z}_X(U) \arrow[r,"\phi_U"] &\mathcal{F}(U) \arrow[r] & 0
\end{tikzcd}\\
\begin{tikzcd} 
\bigoplus_{j\in J}\mathbb{Z}_X|_U \arrow[r,"\varphi"]& \bigoplus_{j\in J}\mathbb{Z}_X|_U \arrow[r,"\phi"] &\mathcal{F}|_U \arrow[r] & 0
\end{tikzcd}
\end{gathered}
\end{equation}
Here $(\varphi_U,\phi_U)=(\oplus_j\varphi_j,\oplus_j\phi_j)$ with $\varphi_U$ and $\phi_U$ the corresponding maps of presheaves, while $\mathcal{F}|_U$ denotes the restricted presheaf with respect to $U$. Since the sheafification operation preserves exactness and commutes with both direct sums and the presheaf restriction, we obtain an exact sequence of sheaves
\[\begin{tikzcd} 
\bigoplus_{j\in J}\widetilde{\mathbb{Z}}_X|_U \arrow[r,"\widetilde{\varphi}"]& \bigoplus_{j\in J}\widetilde{\mathbb{Z}}_X|_U \arrow[r,"\widetilde{\phi}"] &\widetilde{F}|_U \arrow[r] & 0
\end{tikzcd}\]
which holds for every $U$. Clearly this must also be true for any open neighbourhood of an arbitrary point $x$, thus $\widetilde{\mathcal{F}}$ is quasi-coherent. 
\end{proof}

\bibliographystyle{plainurl}
\bibliography{CVPReferences}
\textsc{Department of Mathematics, Texas A\&M University}\par\nopagebreak
\noindent\textit{E-mail address}: \texttt{s.a.k.a.john@tamu.edu}
\end{document}